\documentclass[11pt,a4paper]{article}
\usepackage{amsthm}	
\usepackage{amsfonts, amssymb, amsmath, amscd, latexsym}
\usepackage{eucal, enumitem, latexsym}
\usepackage{graphics}
\usepackage{graphicx}
\usepackage{mathscinet}
\usepackage{psfrag}
\usepackage{tikz}
\usetikzlibrary{decorations.markings}
\usepackage{footmisc}
\usepackage{tikz}
\usepackage{pgfplots}
\pgfplotsset{compat=1.17}
\usepackage{enumitem}
\usepackage{subcaption}
\usetikzlibrary{calc} 
\usepackage{xcolor}

\newcommand{\R}{\mathbb R} 

\newtheorem{theorem}{Theorem}[section]

\newtheorem{proposition}{Proposition}[section]
\newtheorem{definition}{Definition}[section]
\newtheorem{remark}{Remark}[section]
\newtheorem{lemma}{Lemma}[section]

\newcommand{\sgn}{\mathop{\rm sgn}}

\renewcommand{\epsilon}{\varepsilon}
\newcommand{\eps}{\epsilon}
\renewcommand{\phi}{\varphi}
\newenvironment{proofof}[1]{\smallskip\noindent\emph{Proof of #1.}%
\hspace{1pt}}{\hspace{-5pt}{\nobreak\quad\nobreak\hfill\nobreak%
$\square$\vspace{8pt}\par}\smallskip\goodbreak}

\usetikzlibrary{decorations.pathreplacing,decorations.markings}
\tikzset{
  on each segment/.style={
    decorate,
    decoration={
      show path construction,
      moveto code={},
      lineto code={
        \path [#1]
        (\tikzinputsegmentfirst) -- (\tikzinputsegmentlast);
      },
      curveto code={
        \path [#1] (\tikzinputsegmentfirst)
        .. controls
        (\tikzinputsegmentsupporta) and (\tikzinputsegmentsupportb)
        ..
        (\tikzinputsegmentlast);
      },
      closepath code={
        \path [#1]
        (\tikzinputsegmentfirst) -- (\tikzinputsegmentlast);
      },
    },
  },
  mid arrow/.style={postaction={decorate,decoration={
        markings,
        mark=at position .6 with {\arrow[#1]{stealth}}
      }}},
}

\begin{document}
\date{}
\title{
Shock wavefronts for parabolic equations with sign-changing diffusivity
\\
}
\author{
Diego Berti\footnote{Department of Mathematics, University of Turin, Italy}
\and
Andrea Corli\footnote{Department of Mathematics and Computer Science, University of Ferrara, Italy}
\and
Luisa Malaguti\footnote{Department of Sciences and Methods for Engineering, University of Modena and Reggio Emilia, Italy}
}

\maketitle
\begin{abstract}
We consider a reaction–diffusion equation in a one-dimensional space, where the diffusion coefficient changes sign from positive to negative and back to positive. The reaction term is bistable, with its interior zero located in the region where the diffusivity is negative. The model does not admit continuous wavefronts, i.e., continuous traveling-waves that connect the steady states $0$ and $1$. We prove the existence of a family of shock wavefronts, that is, wavefronts with profiles exhibiting a jump discontinuity. We investigate the properties of these profiles and their propagation speeds. Finally, we apply the results to a recently proposed model describing the movement of a population composed of both isolated and grouped individuals.
\end{abstract}

\vspace{1cm}


\section{Introduction}
\label{sec:intro}
In this paper we consider the equation
\begin{equation}\label{e:E}
u_t =P(u)_{xx} + g(u), \qquad t\ge 0, \, x\in \R.
\end{equation}
The unknown function $u$ denotes the density (or concentration) of some quantity and has bounded range; for simplicity we assume $u\in[0,1]$. The diffusion potential $P$ is assigned up to an additive constant; we could take $P(0)=0$, but since the results involve differences of $P$, they are clearer without fixing this choice. About the regularity and structure of the functions $P$ and $g$ we always assume in this paper
\begin{itemize}
\item[(R)] $P\in C^2[0,1]$, $g\in C^0[0,1]$; both $P'$ and $g$ have a finite number of zeros.
\end{itemize}
However, even though the techniques introduced in this paper also apply to the previous general framework, for most of the paper we focus on the case where there exist three real numbers $\alpha,\beta,\gamma$ satisfying
\begin{equation}\label{eq:condition}
0< \alpha < \gamma < \beta <1
\end{equation}
and we assume, in addition to (R), the following conditions, see Figure \ref{f:Dg}:
\begin{itemize}
\item[(P)] $P'>0$ in $(0,\alpha) \cup (\beta,1)$ and $P'<0$ in $(\alpha, \beta)$;

\item[{(g)}] $g(0)=g(1)=0$ and $g<0$ in $(0,\gamma)$, $g>0$ in $(\gamma,1)$.
\end{itemize}

\noindent Condition (P) is equivalent to the following one, expressed through the diffusivity $D:=P'$
\begin{itemize}
\item[]  $D>0$ in $(0,\alpha) \cup (\beta,1)$ and $D<0$ in $(\alpha, \beta)$.
\end{itemize}

\begin{figure}[htb]
\begin{center}
\begin{tikzpicture}[>=stealth, scale=0.6]

\draw[->] (0,0) --  (10,0) node[below]{$u$} coordinate (x axis);
\draw[->] (0,0) -- (0,4) node[right]{$P$} coordinate (y axis);
\draw (0,0) -- (0,-1.5);
\begin{scope}[yshift=20]
\draw[thick] (0,0) .. controls (1,2) and (2,2.5) .. (3,2.5);
\draw[thick] (3,2.5)  .. controls (4,2.5) and (5,1) .. (6,1);
\draw[thick] (6,1)  .. controls (7,1) and (8,2) .. (9,3.5);
\end{scope}

\draw[dotted] (3,0) -- (3,3.2);
\draw[dotted] (6,0) -- (6,1.8);
\draw[dotted] (9,0)  node[below]{{\footnotesize{$1$}}}-- (9,4.2);
\filldraw[black] (3,0) circle (2pt) node[below]{\footnotesize$\alpha$};
\filldraw[black] (6,0) circle (2pt)  node[below]{\footnotesize$\beta$};


\begin{scope}[xshift=350]
\draw[->] (0,0) --  (10,0) node[below]{$u$} coordinate (x axis);
\draw[->] (0,0) -- (0,4) node[right]{$D,g$} coordinate (y axis);
\draw (0,0) -- (0,-1.5);
\draw[thick] (0,1) .. controls (1,2) and (2,2) .. (3,0) node[left=4, below=0]{{\footnotesize{$\alpha$}}} node[midway,above]{$D$};

\draw[thick] (3,0) .. controls (4,-2) and (5,-2) .. (6,0);
\draw[thick] (6,0) node[right=3, below=0]{{\footnotesize{$\beta$}}} .. controls (7,2) and (8,2) .. (9,1);
\draw[dotted] (9,1) -- (9,0) node[below]{{\footnotesize{$1$}}};
\draw[thick,dashed] (0,0) .. controls (1.5,-2) and (3,-2) .. (4.5,0) node[left=4, above=0]{{\footnotesize{$\gamma$}}} node[midway,below]{$g$};
\draw[thick,dashed] (4.5,0) .. controls (6,2) and (7.5,2) .. (9,0);
\filldraw[black] (3,0) circle (2pt);
\filldraw[black] (4.5,0) circle (2pt);
\filldraw[black] (6,0) circle (2pt);
\end{scope}
\end{tikzpicture}

\end{center}
\caption{\label{f:Dg}{Left: plot of the potential $P$. Right: plots of the diffusivity $D$ (solid line) and the source-sink term $g$ (dashed line).}}
\end{figure}
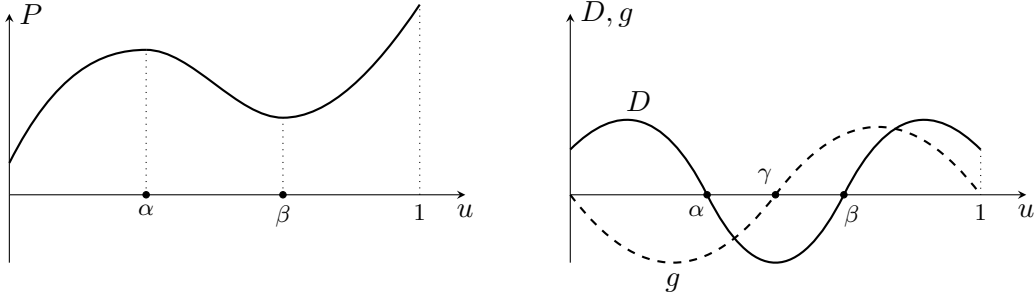
Assumption (P) makes \eqref{e:E} a forward-backward-forward equation; we allow $P'=D$ to possibly vanish at either $0$ or $1$. We recall that a source term $g$ satisfying (g) is called {\em bistable}, and $g$ is called \emph{monostable} if it  is positive in $(0,1)$ and vanishes at $0$ and $1$. Equation \eqref{e:E} is known to govern several biological and physical processes \cite{GK} and it has also been recently proposed in \cite{JBMS} as a model for the movement of a population constituted by isolated and grouped organisms. This model depends on several parameters, and the choice of conditions  (P) and (g) allows  to investigate it under appropriate parameter values (see Section~\ref{s:modello}).

This paper concerns the existence of {\em shock wavefronts} to \eqref{e:E}; precise definitions are provided in Section \ref{e:MR}. 
We recall that a shock wavefront is a solution $u(x,t)=\phi(x-ct)$ to \eqref{e:E}, with $c\in\R$, where the profile $\phi=\phi(\xi)$ is defined on $\R$, monotone, and connects two zeros of $g$. It may have at most a finite number of jump discontinuities and is of class $C^2$ elsewhere except, possibly, at points  where $D$ vanishes. In particular, we focus on the case
\begin{equation} \label{e:phi-limits}
\phi(-\infty)=1\quad \hbox{and}\quad \phi(\infty)=0.
\end{equation}
The case $\phi(-\infty)=0$, $\phi(\infty)=1$ is dealt analogously.
When $\phi$ is continuous, it is called a \emph{regular} wavefront.
The profile must satisfy the ordinary differential equation
\begin{equation}\label{e:EP}
P(\phi)''+c\phi'+g(\phi)=0,
\end{equation}
and solutions are meant in the distributional sense. For regular profiles, \eqref{e:EP} is equivalent to
\begin{equation}\label{e:ED}
\left(D(\phi)\phi'\right)'+c\phi'+g(\phi)=0.
\end{equation}

\paragraph{Regular wavefronts} If $D$ is strictly positive in $(0,1)$, then the existence of regular wavefronts is classical \cite{Fife, GK}. If $D$ changes sign and $g$ is \emph{monostable}, we refer to \cite{BCM2} and \cite{Ferracuti-Marcelli-Papalini}. The case when $D$ changes sign and $g$ is  {\em bistable} falls within  the framework of \cite{JBMS} (see Case 6, sub-case 3, there) and it is treated in  \cite{BCM3} and \cite{Kuzmin-Ruggerini}. 
A key result about such a case is provided in \cite{Kuzmin-Ruggerini}: a wavefront may exist only for a {\em single} value of the speed $c$. If it exists, then
\begin{align}\label{e:KR-cn1}
\text{either}\  \gamma<\alpha, \ c>0,\ \int_0^\alpha D(s)g(s)\, ds>0 \quad \text{ or }\quad \gamma>\beta, \ c<0,\ \int_\beta^1 D(s)g(s)\, ds<0.
\end{align}
As a consequence, if $\gamma\in[\alpha,\beta]$, as it is the case under (P) and (g), then {\em regular wavefronts do not exist}. This motivates our choice of these assumptions. Sufficient conditions for the existence of profiles are also provided in \cite{Kuzmin-Ruggerini}.

\paragraph{Discontinuous wavefronts} Numerical simulations in \cite[p. 13]{JBMS} led the authors to claim that for some $c$, for which regular wavefronts do not exist,
there are shock wavefronts skipping the whole interval $(\alpha,\beta)$, see Figure \ref{f:=}. In this case, equation \eqref{e:EP} has a meaning in the weak sense, that is in the space of distributions while, without imposing further conditions, equation \eqref{e:ED} has not. In fact, if $\phi$ is discontinuous then the same holds for $D(\phi)$, in general; therefore, the term $P(\phi)'=D(\phi)\phi'$ in \eqref{e:ED} would be the product of a discontinuous function with a Dirac mass, which has no distributional meaning.

The existence of discontinuous wavefronts was confirmed numerically in \cite[\S 4.3]{LHMS} and then studied analytically in \cite{Li2021}, for simplified forms of the explicit diffusion and source terms deduced in \cite{JBMS}. The approach in \cite{Li2021}, inspired by \cite{Witelski}, focuses on the analysis of {\em augmented equations}, which are obtained by adding either $-\epsilon^2u_{xxxx}$ or $\epsilon^2u_{xxt}$ to the right-hand side of \eqref{e:E}. The authors in \cite{Li2021} construct regular wavefronts for the augmented equations and then rely on the Geometric Singular Perturbation Theory to pass to the limit $\eps\to0$. The analysis of \cite{Li2021} is continued in \cite{MTMWBH2} where some conditions are introduced to single out the jump; they are, in addition to the continuity of the potential $P$, either the continuity of $D$ or  the equal-area rule for $P$. The latter two conditions coincide in special cases \cite{MTMWBH}; in general, however, they lead to different jumps.

The problem of defining discontinuous solutions to parabolic degenerate equations has a long history. For saturated-diffusion equations we refer to \cite{Bertsch-DalPasso}, where a notion of entropy was introduced to single out solutions. For a slightly different approach 
we refer to \cite{CCCSS_2015, Caselles_entropy, Caselles_flux-limited}. As far as wavefronts are concerned, they can be discontinuous, too, see for instance \cite{CCCSS_qualitative, Chertock-Kurganov-Rosenau_2005,Goodman-Kurganov-Rosenau, Kurganov-Rosenau}. The paper \cite{Campos-Corli-Malaguti} deals with an equation with a convective term and, again, a saturated diffusion. For the case of discontinuous diffusivities we refer to \cite{Drabek-SKZ, Drabek-Zahradinkova21, Drabek-Zahradinkova23}; in \cite{Guarnotta-Marcelli} the reaction term is also discontinuous.

\paragraph{Main results} First, we provide a general framework for dealing with shock wavefronts of equations such as \eqref{e:E}. The result is stated in Proposition~\ref{p:Pconditions}, where only condition (R) is assumed, and it is based on the classical theory of distributions. The main contribution of this paper is Theorem~\ref{t:main}; there, we assume both conditions (P) and (g) and provide a rather complete discussion of the existence of shock wavefronts for equation \eqref{e:E} and their qualitative properties. Our approach does not require the introduction of an augmented equation, but instead relies only on \eqref{e:E}. We build on the analysis in Proposition~\ref{p:Pconditions} and make use of a singular order reduction and the pasting of different stretches of regular profiles, as in \cite{BCM1, BCM2, BCM3, BCM5}. We also exploit results from those papers obtained via comparison methods for lower and upper solutions. These techniques allow us to handle an arbitrary reaction term satisfying (R) and (g), without requiring any specific polynomial structure, as, for instance, in \cite{MTMWBH2}. We can further determine the admissible speeds (see \eqref{e:speed}) and investigate their properties as a function of the jump.


Here follows our main result; throughout, {\em uniqueness} of profiles is always meant up to horizontal shifts. We refer to Figure \ref{f:=}.

\begin{theorem}
\label{t:main}
Assume conditions {\rm (R)}, {\rm (P)} and {\rm (g)}; then we have the following results.

\begin{enumerate}
\item[(1)]
Assume $P(1)>P(0)$.  There exists a closed non-empty interval $\mathcal I\subset [0,\alpha]$ such that, for every $\varphi_\ell \in \mathcal I$, there are unique $\varphi_r\in[\beta,1]$ and $c^*\in \mathbb R$ for which \eqref{e:E} has a unique shock wavefront satisfying \eqref{e:phi-limits}, with speed $c^*$, whose profile $\varphi$ has a single jump from $\varphi_r$ to $\varphi_\ell$. Moreover, $\varphi_r$ is determined by
\begin{equation}
\label{e:condition_P}
P(\varphi_r)=P(\varphi_\ell),
\end{equation}  
and $c^*(\varphi_\ell)$ is continuous and strictly decreasing. Equation \eqref{e:E} admits no other shock wavefronts satisfying \eqref{e:phi-limits}.

\item[(2)] Assume $P(1)=P(0)$. Then Equation \eqref{e:E} admits the piecewise constant shock wavefront satisfying \eqref{e:phi-limits} whose profile jumps from $1$ to $0$ and has speed $c=0$. If, in addition, $P(\gamma)=P(0)$, then \eqref{e:E} also admits the piecewise constant shock wavefront whose profile has two jumps, one from $1$ to $\gamma$ and the other from $\gamma$ to $0$; the speed is again $0$. 
Equation \eqref{e:E} has no other shock wavefronts satisfying \eqref{e:phi-limits}.

\item[(3)] Assume $P(1)<P(0)$. Then Equation \eqref{e:E} has no shock wavefronts satisfying \eqref{e:phi-limits}.
\end{enumerate}
\end{theorem}

\begin{figure}[htb]
\begin{center}
\begin{tikzpicture}[>=stealth, scale=0.65]

\draw[->] (0,0) --  (10,0) node[below]{$u$} coordinate (x axis);
\draw[->] (0,0) -- (0,4.5) node[right]{$P$} coordinate (y axis);
\draw (0,0) -- (0,-1);

\begin{scope}[yshift=20]
\draw[very thick] (0,0) .. controls (1,2) and (2,2.5) .. (3,2.5);
\draw[very thick] (3,2.5)  .. controls (4,2.5) and (5,1) .. (6,1);
\draw[very thick] (6,1)  .. controls (7,1) and (8,2) .. (9,3.5);
\draw[dotted](0,3.5) node[left]{\footnotesize$P(1)$}-- (9,3.5);
\draw[dashed, thick] (1,1.5)--(7.25,1.5);
\end{scope}
\draw[dotted] (3,-1) -- (3,3.2);
\draw[dotted] (6,0) -- (6,1.8);
\draw[dotted] (9,0)  node[below]{{\footnotesize{$1$}}}-- (9,4.2);
\filldraw[black] (3,0) circle (2pt) node[below]{\footnotesize$\alpha$};
\filldraw[black] (6,0) circle (2pt)  node[below=7, left=-1]{\footnotesize$\beta$};

\draw[dotted] (0.55,1.68)--(6,1.68);
\draw[dotted] (0.58,-1)--(0.58,1.68);
\draw[<->] (0.55,-1)--(3,-1) node[midway,below]{$\mathcal{I}$};

\draw[dotted] (1,0) node[below]{\footnotesize$\phi_\ell$}--(1,2.2);
\draw[dotted] (7.25,0) node[below]{\footnotesize$\phi_r$}--(7.25,2.2);

\begin{scope}[xshift=420]
\draw[->] (-4,0) --  (6,0) node[below]{$\xi$} coordinate (x axis);
\draw[->] (0,0) -- (0,4.5) node[right]{$\phi$} coordinate (y axis);

\draw (-4,3.5) -- (6,3.5);
\draw[very thick] (-4,3.3) .. controls (-2,3.3) and (0,3.2).. (2,2.7) node[near start,below]{\footnotesize$\phi$};
\draw[very thick] (2,1) .. controls (3,0.5) and (5,0.2) ..  (6,0.2) node[near end,above]{\footnotesize$\phi$};

\draw(-0.1,2.3) node[left]{\footnotesize$\beta$}--(0.1,2.3);
\draw(-0.1,1.4) node[left]{\footnotesize$\alpha$}--(0.1,1.4);

\draw[dotted] (2,0) node[below]{$\xi_s$} --(2,2.7);
\draw[dotted] (0,2.7) node[left]{\footnotesize$\phi_r$}--(2,2.7);
\draw[dotted] (0,1) node[left]{\footnotesize$\phi_\ell$}--(2,1);

\draw (0,4) node[left]{$1$};

\end{scope}

\begin{scope}[yshift=-200]
\begin{scope}[yshift=50]
\draw[->] (0,0) --  (9.5,0) node[below]{$u$} coordinate (x axis);
\draw[->] (0,0) -- (0,2.5) node[right]{$P$} coordinate (y axis);
\draw (0,0) -- (0,-0.5);

\begin{scope}[yshift=20]
\draw[very thick] (0,0) .. controls (1,1) and (2,1.5) .. (3,1.5);
\draw[very thick] (3,1.5)  .. controls (4.7,1.5) and (4.7,-1.5) .. (6,-1.5);
\draw[very thick] (6,-1.5)  .. controls (7,-1.5) and (8,-1) .. (9,0);
\draw[thick, dashed](0,0) node[left]{\footnotesize$P(1)=P(0)$}-- (9,0);
\end{scope}
\draw[dotted] (3,0) -- (3,2.2);
\draw[dotted] (4,0) -- (4,1.8);
\draw[dotted] (6,0) -- (6,-0.8);
\draw[dotted] (9,0)  node[below]{{\footnotesize{$1$}}}-- (9,0.7);
\filldraw[black] (3,0) circle (2pt) node[below]{\footnotesize$\alpha$};
\filldraw[black] (4,0) circle (2pt) node[below]{\footnotesize$\gamma$};
\filldraw[black] (6,0) circle (2pt)  node[above]{\footnotesize$\beta$};
\end{scope}

\begin{scope}[yshift=-50]
\draw[->] (0,0) --  (9.5,0) node[below]{$u$} coordinate (x axis);
\draw[->] (0,0) -- (0,2.5) node[right]{$P$} coordinate (y axis);
\draw (0,0) -- (0,-0.5);
\begin{scope}[yshift=20]
\draw[very thick] (0,0) .. controls (1,1) and (2,1.5) .. (3,1.5);
\draw[very thick] (3,1.5)  .. controls (4.7,1.5) and (4.7,-1.5) .. (6,-1.5);
\draw[very thick] (6,-1.5)  .. controls (7,-1.5) and (8,-1) .. (9,0);
\draw[thick, dashed](0,0) node[left]{\footnotesize$P(1)=P(0)$}-- (9,0);
\end{scope}
\draw[dotted] (3,0) -- (3,2.2);
\draw[dotted] (4.62,0) -- (4.62,0.8);
\draw[dotted] (6,0) -- (6,-0.8);
\draw[dotted] (9,0)  node[below]{{\footnotesize{$1$}}}-- (9,0.7);
\filldraw[black] (3,0) circle (2pt) node[below]{\footnotesize$\alpha$};
\filldraw[black] (4.62,0) circle (2pt) node[below]{\footnotesize$\gamma$};
\filldraw[black] (6,0) circle (2pt)  node[above]{\footnotesize$\beta$};
\end{scope}

\begin{scope}[xshift=420,yshift=50]
\draw[->] (-4,0) --  (6,0) node[above]{$\xi$} coordinate (x axis);
\draw[->] (0,0) -- (0,3) node[right]{$\phi$} coordinate (y axis);
\draw (-4,2.5) -- (6,2.5);
\draw[very thick] (-4,2.5) -- (2,2.5) node[near start,below]{\footnotesize$\phi$};
\draw[very thick] (2,0) -- (6,0) node[near end,above]{\footnotesize$\phi$};
\draw[dotted] (2,0) node[above=8, right=0]{$\xi_s$} --(2,2.7);
\draw (0,3) node[left]{$1$};
\end{scope}

\begin{scope}[xshift=420,yshift=-50]
\draw[->] (-4,0) --  (6,0) node[below]{$\xi$} coordinate (x axis);
\draw[->] (0,0) -- (0,3) node[right]{$\phi$} coordinate (y axis);
\draw (-4,2.5) -- (6,2.5);
\draw[very thick] (-4,2.5) -- (-2,2.5) node[near start,below]{\footnotesize$\phi$};
\draw[very thick] (-2,1.4) -- (2,1.4) node[near start,below]{\footnotesize$\phi$};
\draw[very thick] (2,0) -- (6,0) node[near end,below]{\footnotesize$\phi$};
\draw[dotted] (-2,0) node[below]{$\xi_{*,1}$} --(-2,2.5);
\draw[dotted] (2,0) node[below]{$\xi_{*,2}$} --(2,1.4);
\draw (0,3) node[left]{$1$};
\draw (0,1.4) node[right=6,above=0]{\footnotesize$\gamma$};
\end{scope}
\end{scope}

\end{tikzpicture}

\end{center}
\caption{\label{f:=}{Above: the case $P(1)>P(0)$ and the interval $\mathcal{I}$. Middle: the case $P(1)=P(0)$. Below: the case $P(1)=P(\gamma)=P(0)$.}}
\end{figure}
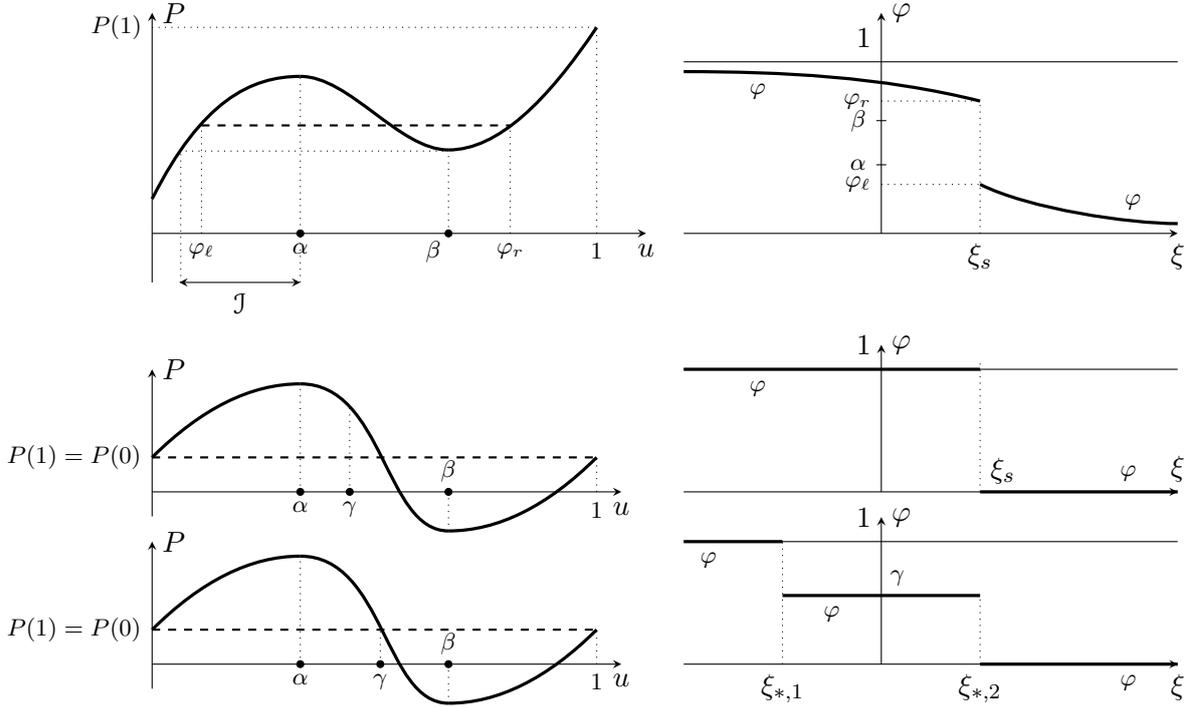

Every shock wavefront profile satisfying condition \eqref{e:phi-limits} is decreasing, see its definition in Section \ref{e:MR}; in general, this monotonicity is not strict. This happens, in particular, when either $0 \in \mathcal I$, or $\varphi_r=1$ satisfies \eqref{e:condition_P}, for some $\varphi_\ell \in \mathcal I$. This is immediate if $P(1)=P(0)$, but it also holds  when $P(1)>P(0)$. More precisely, if $P(\beta)\le P(0)<P(1)$, then equation \eqref{e:E} admits a shock wavefront such that $\phi(\xi)\equiv0$ for $\xi > \xi_s$, for some real $\xi_s$. Similarly, if $P(0)<P(1)\le P(\alpha)$, then \eqref{e:E} has a shock wavefront with $\phi(\xi)\equiv1$ for $\xi \le \xi_s$, for some real $\xi_s$. See Figure \ref{f:hatcheck-rem}.

\begin{figure}[htb]
\begin{center}
\begin{tikzpicture}[>=stealth, scale=0.65]

\draw[->] (-4,0) --  (6,0) node[below]{$\xi$} coordinate (x axis);
\draw[->] (0,0) -- (0,4.5) node[right]{$\phi$} coordinate (y axis);
\draw (-4,3.5) -- (6,3.5);
\draw[very thick] (-4,3.3) .. controls (-2,3.3) and (0,3.2).. (2,2.7) node[near start,below]{\footnotesize$\phi$};
\draw[very thick] (2,0) -- (6,0) node[near end,above]{\footnotesize$\phi$};
\draw[dotted] (2,0) node[below]{\footnotesize$\xi_s$} --(2,2.7);
\draw[dotted] (0,2.7) node[left]{\footnotesize$\phi_r$}--(2,2.7);
\draw (0,4.5) node[left]{$1$};
\draw(-0.1,2.3) node[left]{\footnotesize$\beta$}--(0.1,2.3);
\draw(-0.1,1.4) node[left]{\footnotesize$\alpha$}--(0.1,1.4);

\begin{scope}[xshift=350]
\draw[->] (-4,0) --  (6,0) node[below]{$\xi$} coordinate (x axis);
\draw[->] (0,0) -- (0,4.5) node[right]{$\phi$} coordinate (y axis);
\draw (-4,3.5) -- (6,3.5);
\draw[very thick] (2,1) .. controls (3,0.5) and (5,0.2) ..  (6,0.2) node[near end,above]{\footnotesize$\phi$};
\draw[very thick] (-4,3.5) -- (2,3.5) node[near start,below]{\footnotesize$\phi$};
\draw[dotted] (2,0) node[below]{\footnotesize$\xi_s$} --(2,3.5);
\draw[dotted] (0,1) node[left]{\footnotesize$\phi_\ell$}--(2,1);
\draw (0,4.5) node[left]{$1$};
\draw(-0.1,2.3) node[left]{\footnotesize$\beta$}--(0.1,2.3);
\draw(-0.1,1.4) node[left]{\footnotesize$\alpha$}--(0.1,1.4);

\end{scope}\end{tikzpicture}

\end{center}
\caption{\label{f:hatcheck-rem}{The profile $\phi$ when $P(\beta)\le P(0)<P(1)$ (left) and $P(0)<P(1)\le P(\alpha)$ (right).}}
\end{figure}
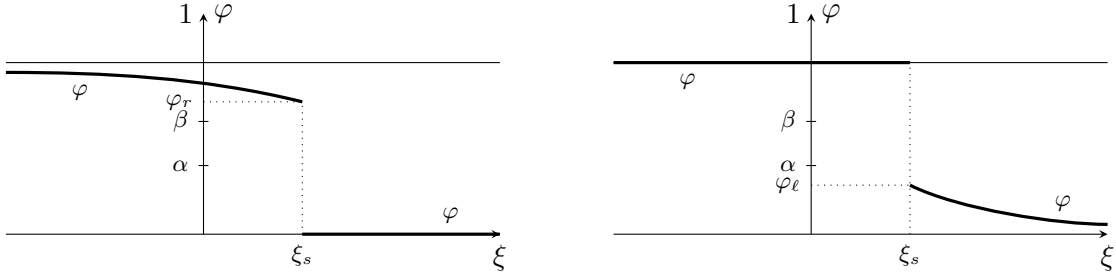

If $P(1)>P(0)$, then $c^*(\varphi_\ell)$ is continuous on the closed interval $\mathcal I$, and hence it is bounded. We refer to Proposition~\ref{prop:c*_bounds} for estimates of its bounds. In Proposition~\ref{prop:necess2} we show that 
\begin{equation}\label{e:phireta}
\varphi_r = \eta(\varphi_\ell)
\end{equation} 
for some strictly increasing function $\eta \in C^0(\mathcal I) \cap C^1(\mathrm{int}(\mathcal I))$. In Remark \ref{r:skkwe} we establish, in some cases, whether $c^*(\varphi_\ell)>0$ or $c^*(\varphi_\ell)<0$ for some $\varphi_\ell\in\mathcal{I}$.

The constraint \eqref{e:condition_P} was also required in \cite{Li2021, MTMWBH2}. A second constraint is expressed by \eqref{e:hypo}$_2$, namely
\begin{equation}\label{e:speed}
c^*(\phi_\ell)=-\frac{\displaystyle P(\varphi)'(\xi_s^-)-P(\phi)'(\xi_s^+)}{\phi(\xi_s^-)-\phi(\xi_s^+)}, \quad \phi_\ell=\phi(\xi_s^-), \, \phi_r= \phi(\xi_s^+),
\end{equation}
where $\xi_s$ is the jump point. All terms appearing in the previous estimate exist and the condition requires that the quantity $P'(\varphi) + c\varphi$ is continuous  across the shock (see Proposition~\ref{p:Pconditions}). Thus, it closely resembles the classical Rankine–Hugoniot condition for shocks in the hyperbolic setting (see Subsection~\ref{ss:concluding-remark}).

As already pointed out in \cite{Li2021, MTMWBH2}, condition \eqref{e:condition_P} does not uniquely determine the jump states $\phi_\ell$ and $\phi_r$. Since there are two degrees of freedom, namely $\phi_\ell$ and $\phi_r$,  an additional condition is generally required to uniquely single out the profiles.
Non-uniqueness of wavefronts also occurs when $P'\ge 0$ and $g$ is monostable. In this case, the profiles are regular and the admissible speeds form the interval $[c^*, \infty)$. However, the profile with minimal speed $c^*$ enjoys suitable stability properties (see \cite[\S 4.5]{Aronson-Weinberger}). In analogy with the previous case, one may adopt here the criterion of minimizing the speed $c^*$. This is possible because $c^*(\phi_\ell)$ is a strictly decreasing and continuous function on the closed interval $\mathcal{I}$ (see Theorem \ref{t:main}).
In particular  $c^*(\phi_{\ell})$  is minimum at 
\[
\phi_\ell^*=\max\mathcal{I}
=
\left\{
\begin{array}{ll}
\alpha & \hbox{ if }P(\alpha)<P(1),
\\
\eta^{-1}(1)
& \hbox{ if } P(\alpha)>P(1).
\end{array}
\right.
\]
When selecting the minimal-speed wavefront, its corresponding profile remains in $[0,\alpha]$ for as long as possible, and jumps otherwise.

\noindent Another possibility is to further  assume, as in \cite{MTMWBH2}, the so called \emph{continuous diffusivity rule}: $D(\phi_{\ell})=D(\phi_r)$. When $P(\alpha)\le P(1)$, this condition is always satisfied by a unique pair $(\phi_{\ell}, \phi_r)$ fulfilling \eqref{e:condition_P} (see \cite[Proposition 2]{MTMWBH2}). However, such a pair may fail to exist if  $P(\alpha)> P(1)$ as we will show in Proposition \ref{p:emm} (see also Remark \ref{r:P1alpha} and Figure  \ref{f:inter}).

The assumption in Theorem \ref{t:main} can be expressed in terms of $D$. In particular, we refer to Figure \ref{f:Dcont} for the geometrical interpretation of condition \eqref{e:condition_P}. 
No shock wavefront satisfying \eqref{e:phi-limits} can exist when $\int_{0}^{1}D(u) \, du<0$. Instead, when $\int_{0}^{1}D(u) \, du=0$, equation \eqref{e:E} has shock wavefronts satisfying \eqref{e:phi-limits} only with speed $c=0$ and profile $\phi$ with $\phi'=0$ a.e.

\begin{figure}[htb]
\begin{center}
\begin{tikzpicture}[>=stealth, scale=0.8]
\draw[->] (0,0) --  (10,0) node[below]{$u$} coordinate (x axis);
\draw[->] (0,0) -- (0,3) node[right]{$D$} coordinate (y axis);
\draw (0,0) -- (0,-1.5);
\draw[thick] (0,1) .. controls (1,2) and (2,2) .. (3,0) node[left=3, below=0]{{\footnotesize{$\alpha$}}}; 
\draw[thick] (3,0) .. controls (4,-2) and (5,-2) .. (6,0); 

\filldraw[fill=gray!20!cyan, draw=black, nearly transparent] (3,0) .. controls (4,-2) and (5,-2) .. (6,0)  -- (3,0)--cycle;

\draw[thick] (6,0) node[right=2, below=0]{{\footnotesize{$\beta$}}} .. controls (7,2) and (8,2) .. (9,1); 
\draw[dotted] (9,1) -- (9,0) node[below]{{\footnotesize{$1$}}};
\draw[dotted] (1.7,1.5)--(6.7,1);
\draw[dotted] (1.7,0) node[below]{\footnotesize$\phi_\ell$}-- (1.7,1.5);
\filldraw[fill=gray!20!yellow, draw=black, nearly transparent] (1.7,1.55) .. controls (2.5,1) and (2.7,0.5) .. (3,0)  -- (1.7,0) -- (1.7,1.55)--cycle;

\draw[dotted] (6.6,0) node[below]{\footnotesize$\phi_r$}-- (6.6,1);
\filldraw[fill=gray!20!yellow, draw=black, nearly transparent] (6,0) .. controls (6.5,1) and (6.7,1) .. (6.6,1)  -- (6.6,0) -- (6,0)--cycle;
\end{tikzpicture}

\end{center}
\caption{\label{f:Dcont}{Condition \eqref{e:condition_P} states that the area of the yellow region equals that of the blue region.}}
\end{figure}
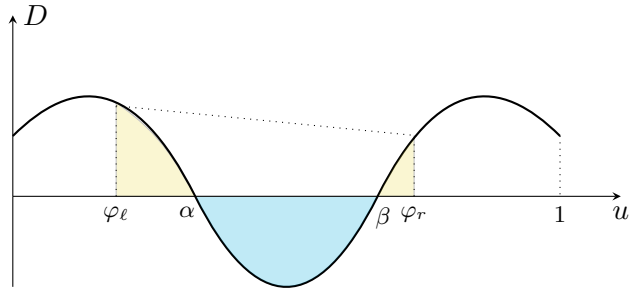

 Theorem \ref{t:main} holds, with the additional assumption $P(\alpha)>P(\beta)$, by only assuming that $P$ is defined in $[0,\alpha]\cup[\beta,1]$ but not necessarily in $(\alpha,\beta)$. The further assumption that $P$ is defined and decreasing in $(\alpha,\beta)$ allows us, on the one hand, to state the equal area condition for $D$. On the other hand, it rules out the possibilities of other shock wavefronts valued in $(\alpha,\beta)$, see Lemma \ref{l:signC} and Proposition \ref{prop:necess}.

Here follows an outline of the paper. In Section \ref{e:MR} we provide the main definitions and prove some properties of shock traveling waves under the only assumption (R). Section \ref{s:RSWFs} focuses on regular semi-wavefronts. Section \ref{s:WF} contains the proof of Theorem \ref{t:main} as long as several related results. In particular, Subsection \ref{ss:speed} provides some estimates for the lower and upper bound of $c^*$ while Subsection \ref{ss:concluding-remark} discusses some relationships with balance laws.
\section{An invasion model involving isolated and grouped agents}\label{s:modello}
\setcounter{equation}{0}
As observed in the Introduction, one of the motivations for studying equation \eqref{e:E}, with $P$ and $g$ satisfying (R) and (P)--(g), is the biological model introduced in \cite{JBMS}. There, the diffusivity $D$ and the reaction term $g$ depend upon the six positive parameters $D_{i,g}$, $k_{i,g}$, $\lambda_{i,g}$ and have the following explicit expressions:
\begin{align}
\label{e:coeffs D}
D(u) &= D_i\left(1-4u+3u^2\right) + D_g\left(4u-3u^2\right),
\\
\label{e:coeffs g}
g(u) &= \lambda_g u (1-u) +\left[\lambda_i-\lambda_g- \left(k_i - k_g\right) \right] u(1-u)^2 -k_g u.
\end{align}

\begin{lemma}[\cite{BCM5}]\label{l:la}
The potential $P$ with $P'=D$ in \eqref{e:coeffs D} satisfies {\rm (P)} if and only if $D_i>4D_g>0$. In this case we have
\begin{equation}\label{e:alphabeta}
\alpha=\frac23 - \frac\omega3 \quad \hbox{ and }\quad \beta=\frac23 + \frac\omega3, \qquad  \hbox{ for }\  \omega :=\sqrt{\frac{D_i-4D_g}{D_i-D_g}}\in (0,1).
\end{equation}
The reaction term $g$ in \eqref{e:coeffs g} satisfies {\rm (g)} with \eqref{eq:condition} if and only if 
\begin{equation}\label{e:acb}
k_g=0, \, \lambda_g>0 \ \mbox{ and } \ \frac{1-\omega}{2+\omega} < \frac{\lambda_g}{k_i-\lambda_i}<\frac{1+\omega}{2-\omega}.
\end{equation}
\end{lemma}

Using Theorem \ref{t:main}, we can now characterize the presence of shock wavefronts in the model defined by \eqref{e:coeffs D}–\eqref{e:coeffs g}. Because of the parameter $\omega$, we denote by $\mathcal{I}_\omega$ the interval $\mathcal{I}$ in the statement of Theorem \ref{t:main}.

\begin{proposition}
\label{prop:A}
Consider  \eqref{e:coeffs D}-\eqref{e:coeffs g}, with $D_i>4D_g>0$ and \eqref{e:acb}. 
Assume 
\begin{enumerate}
\item
either $\omega \in (0,1/2]$ and
\[
\varphi_\ell \in \mathcal{I}_\omega:= \left[\alpha - \frac{\omega}{3}, \alpha\right], \quad \varphi_r = \eta(\varphi_\ell) \in \left[ \beta, \beta+\frac{\omega}{3}\right];
\]
\item or $\omega \in (1/2,1)$ and 
\[
\varphi_\ell \in  \mathcal{I}_\omega:=\left[\alpha -\frac{\omega}{3}, \frac12 -\frac{\sqrt{4\omega^2-1}}{2\sqrt{3}} \right]\subset (0,\alpha), \quad  \varphi_r =\eta(\varphi_\ell)\in [\beta,1].
\]
\end{enumerate}
Then, for every $\phi_{\ell}\in \mathcal{I}_\omega$,  Equation \eqref{e:E} admits a unique shock wavefront satisfiyng \eqref{e:phi-limits}, with speed $c^*(\varphi_\ell)$; its profile jumps from $\varphi_r$ to $\varphi_\ell,  \, c^*(\varphi_\ell)$ is continuous and strictly decreasing and
\begin{equation}
\label{eta_J}
\eta(u):=\frac{2-u+\sqrt{\Delta_u}}{2} \quad \mbox{ with } \quad 
\Delta_u:= -3u^2+4u-\frac 43 +\frac{4\omega^2}{3}.
\end{equation} 
Equation \eqref{e:E} admits no other shock or regular wavefront satisfying \eqref{e:phi-limits}.

\end{proposition}

We refer to Figure \ref{f:triangle} for an illustration of Proposition \ref{prop:A}.

\begin{figure}[htbp]
    \centering
    \begin{tabular}{cc}
    \includegraphics[width=6.5cm]{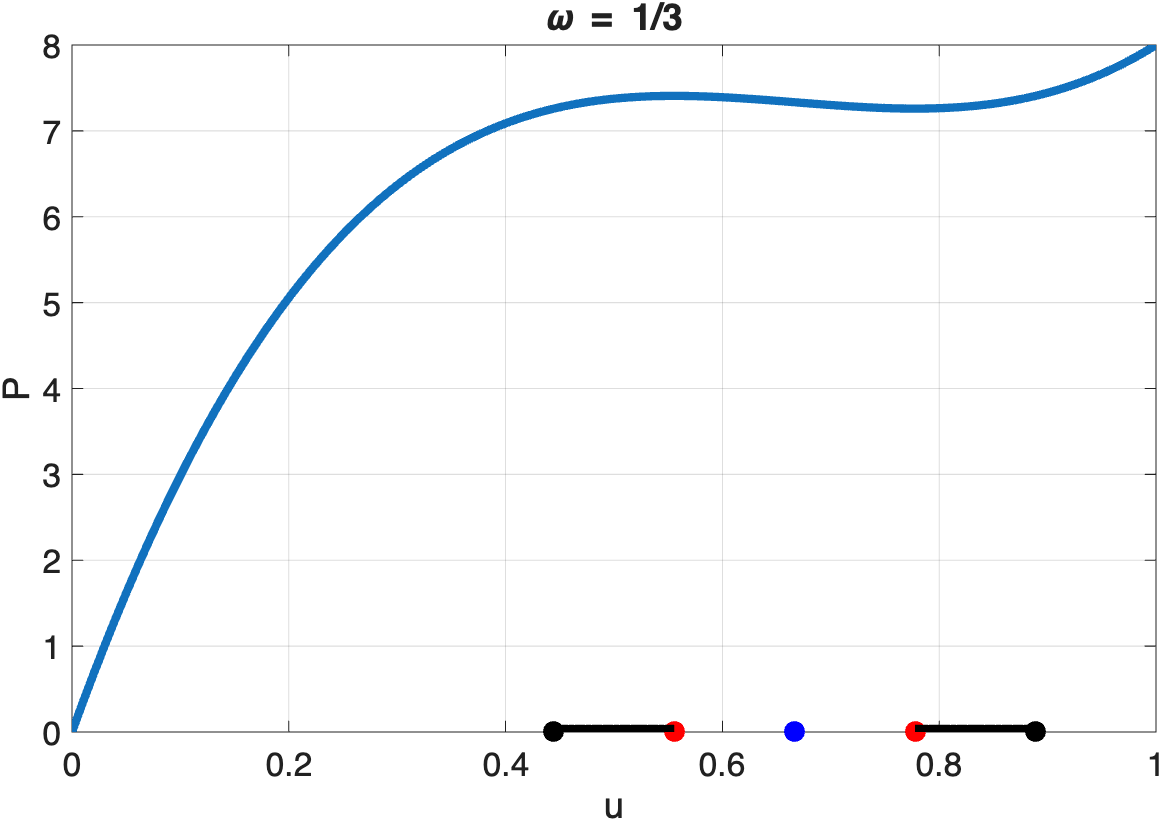}
    &
    \includegraphics[width=6.5cm]{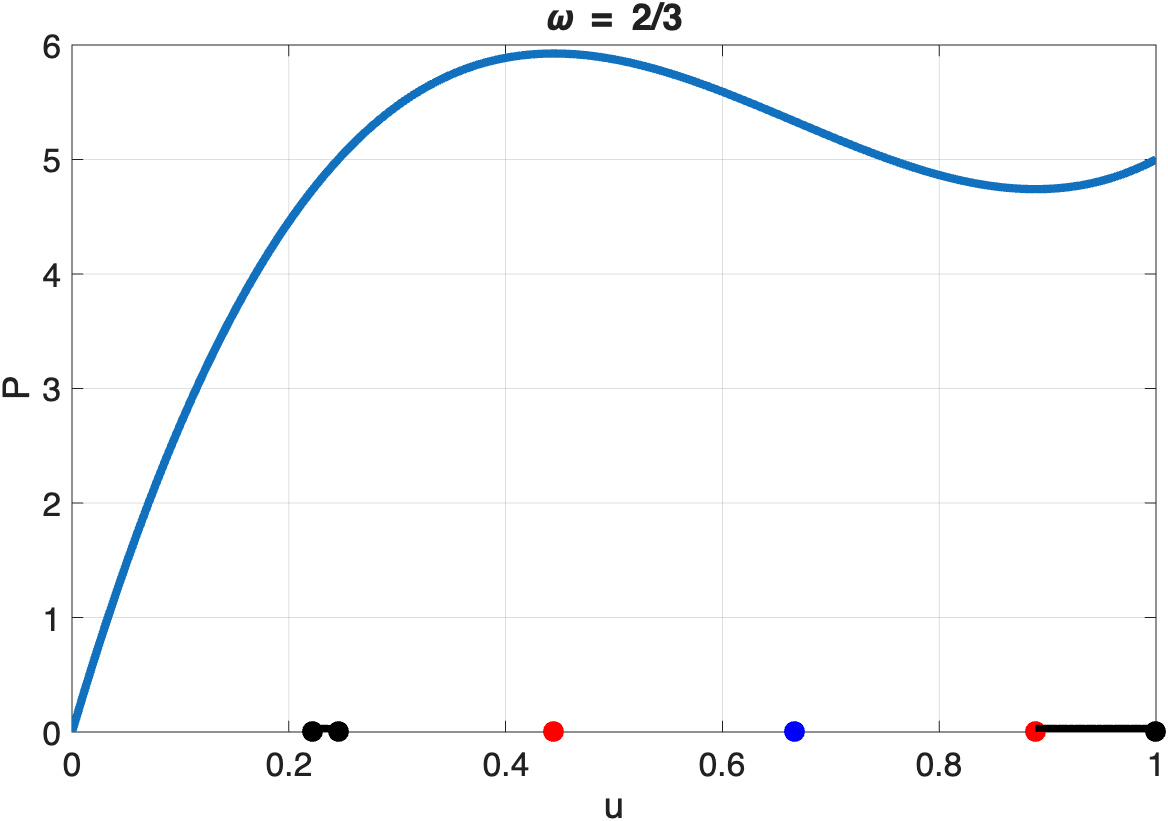}
    \end{tabular}
 \caption{Plots of $P$. Red points denote $\alpha$ and $\beta$, the blue point is the point $(\frac23,0)$, black points are the extrema of the intervals $\mathcal{I}$ and $\eta(\mathcal{I})$. 
 On the left:  $D_i=35$, $D_g=8$, so that $\omega = 1/3$. On the right, $D_i=32$, $D_g=5$ so that $\omega=2/3$.
}
    \label{f:triangle}
    \end{figure}

\begin{proof}
We rewrite formula \eqref{e:coeffs D} by exploiting \eqref{e:alphabeta}; we obtain 
\begin{align}
D(u)&=3(D_i-D_g)\left(u^2 -(\alpha+\beta)u+\alpha\beta \right)
= (D_i-D_g)\left(3u^2 -4 u+\frac{4-\omega^2}{3}\right),
\label{e:DDD}
\\
P(u)&=(D_i-D_g)\left(u^3 - 2u^2+\frac{4-\omega^2}{3}u \right)=(D_i-D_g)u\left(u^2 - 2u+\frac{4-\omega^2}{3}\right),
\nonumber
\end{align}
by choosing $C=0$ as integration constant. Notice that 
\begin{equation}\label{e:eoer}
u^2 - 2u+\frac{4-\omega^2}{3}>0, \quad u\in\R.
\end{equation}
%
Therefore $P(1)=(D_i-D_g)\frac{1-\omega^2}{3}>P(0)=0$ and Theorem \ref{t:main} {\em Item 1} applies.
\par

It remains to write down the explicit expressions of $\mathcal I_\omega$ and $\varphi_r=\varphi_r(\varphi_\ell)$. Notice that
\begin{equation} 
\label{e:A}
\begin{aligned}
\frac{P(u)-P(v)}{D_i-D_g}&
=(u-v)\left(u^2+uv+v^2-2(u+v)+ \frac{4-\omega^2}{3}\right).
\end{aligned}
\end{equation}
It follows that $P(u)=P(v)$ if and only if
	\begin{equation}
	\label{e:AAB}
	v^2+v(u-2)+\left(u^2-2\,u+\frac{4-\omega^2}{3}\right)=0.
	\end{equation}
Then \eqref{e:AAB} admits solutions if and only if $\Delta_u \ge 0$, see \eqref{eta_J}, which is equivalent to $u \in \left[\frac{2-2\omega}{3}, \frac{2+2\omega}{3}\right]$.
Intersecting with the interval $[0,\alpha]$, this gives
\[
\mathcal I_0:=\left[\Delta_u \ge 0\right] \cap [0,\alpha]= \left[\frac{2-2\omega}{3}, \frac{2-\omega}{3}\right]=\left[\alpha-\frac{\omega}{3}, \alpha\right].
\]

For every $u \in \mathcal I_0$, we can solve \eqref{e:AAB} with respect to $v$, obtaining
\[
v_\pm=\frac{2-u\pm \sqrt{\Delta_u}}{2}.
\]
For $u \in \mathcal I_0$ with $u>\frac{2-2\omega}{3}$, we have $\Delta_u>0$ and $
v_-=\frac{2-u-\sqrt{\Delta_u}}{2}<\frac{2-u}{2}< \frac{2-\frac{2-2\omega}{3}}{2}=\frac{2+\omega}{3}=\beta$.
At $u=\frac{2-2\omega}{3}$ one has $\Delta_u=0$ and $v_-=v_+=\beta$.
So the unique solution in $[\beta,1]$ of \eqref{e:AAB} is  
$v_+=\eta(u)$, with $\eta$ as in \eqref{eta_J}.
Direct computations show that $\eta(u)$ is strictly increasing. This property also holds in the general case (see Proposition \ref{prop:necess2}). Hence, if $u \in \mathcal I_0$, then 
\[
\eta(u) \in \left[\eta\left(\frac{2-2\omega}{3}\right), \eta\left(\alpha \right)\right] = \left[\frac{2+\omega}{3},\frac{2+2\omega}{3}\right]=\left[\beta, \beta+\frac{\omega}{3}\right].
\]
Since the last interval is contained in $[0,1]$ iff $\omega \in (0,1/2]$, an application of Theorem \ref{t:main} concludes the proof in this case.

In the case $\omega > 1/2$ we have $\frac{2+2\omega}{3}>1$. We need to find $u^* \in\left[\frac{2-2\omega}{3}, \alpha\right]$ such that
\[
\mathcal I_\omega = \left[\frac{2-2\omega}{3},\, u^*\right]
\qquad \text{and} \qquad
\eta(\mathcal I) = \left[\beta,\, 1\right].
\]
We claim that $u^*=\frac12-\frac{\sqrt{4\omega^2-1}}{2\sqrt{3}}$. Surely $u^*$ belongs to the above interval. By the strict monotonicity of $\eta$, the desired $u^*$ does exist, and moreover it must be the unique value in $[0,\alpha]$ such that
$1=\eta(u^*)$, which is equivalent to $u^* =\sqrt{\Delta_{u^*}}$.
The solutions of $s^2=\Delta_s$ are the $s$ such that
\[
s^2-s+\frac{1-\omega^2}{3}=0 \iff s_\pm
=\frac12\pm \frac{\sqrt{4\omega^2-1}}{2\sqrt{3}}.
\]
Since $s_+ > \alpha$, then $u^*=s_-$. The proof is completed by an application of Theorem \ref{t:main}.
\end{proof}

\begin{remark}\label{r:P1alpha}
{
\rm
By \eqref{e:A} we deduce
$
P(1)-P(\alpha)
=-(D_i-D_g)(1+\omega)\left(2\omega^2+\omega-1\right)/27$.
Then $P(1)\ge P(\alpha)$ iff $2\omega^2+\omega-1\le0$, i.e., iff $-1\le\omega\le\frac12$. 
Moreover $P(\beta)>P(0)=0$ by \eqref{e:eoer}. As a conclusion, Proposition \ref{prop:c*_bounds} applies if $\omega\le\frac12$, so that estimates \eqref{e:c1c2} for $c^*$ hold.
}
\end{remark}

We already observed in the Introduction that uniqueness can be recovered from the intrinsic continuum of admissible jumps (and speeds) by imposing algebraic conditions. In the following proposition, we focus on the constraint
\begin{equation}
\label{e:D_cont}
D(\varphi_\ell)=D(\varphi_r).
\end{equation}
Recall that $D=P'$. Figure~\ref{f:triangle} illustrates a case where the constraint is satisfied (left) and one where it is not (right). More generally, the following result holds.
\begin{proposition}\label{p:emm}
Consider \eqref{e:coeffs D}-\eqref{e:coeffs g}, with $D_i>4D_g>0$ and \eqref{e:acb}. If $\omega \in (0, 1/\sqrt{3}]$, then there is a unique shock wavefront satisfying \eqref{e:phi-limits}  and \eqref{e:D_cont}; if $\omega \in (1/\sqrt{3},1)$, there are none.

\end{proposition}

\begin{proof}
By \eqref{e:DDD} we see that \eqref{e:D_cont} is equivalent to 
\begin{equation}\label{e:akdd}
\left(u-\eta(u)\right)\left(u+\eta(u)-(\alpha+\beta)\right) = 0,
\end{equation}
for $\eta$ as in \eqref{eta_J}, that is $u+\eta(u)-4/3=0$. From \eqref{eta_J}$_1$ we obtain $-u+\frac 23 =\sqrt{\Delta_u}$. Notice that $u \in \mathcal I_{\omega}$ implies $u\le \alpha=\frac 23-\frac{\omega}{3}$, and hence $-u+\frac 23>0$. Using \eqref{eta_J}$_2$ we get $9u^2-12 u+4-3\omega^2=0$, which yields 
\[
u=\frac 23 -\frac{\omega}{\sqrt 3},
\]
since the other root does not belong to $\mathcal I_{\omega}$. It remains to verify that $\frac 23 -\frac{\omega}{\sqrt 3} \in \mathcal I_\omega$.
This holds whenever $\omega \in (0, 1/2)$, since in this case  $\mathcal I_\omega =  \left[\frac 23 -\frac{2\omega}{3},\, \frac 23 -\frac{\omega}{3}\right]$. 
On the other hand, when $\omega \in (1/2, 1)$, we have $\mathcal I_\omega=\left[\frac 23 -\frac{2\omega}{3},\, \frac 12 -\frac{\sqrt{4\omega^2 -1}}{2\sqrt 3}\right]$. Moreover  the estimate 
\[
\frac 23 -\frac{\omega}{\sqrt 3}\le \frac 12 -\frac{\sqrt{4\omega^2 -1}}{2\sqrt 3}
\]
holds  only for $\omega \le \frac{1}{\sqrt 3}$.
\end{proof}
\begin{figure}[htbp]
    \centering
    \begin{tabular}{cc}
    \includegraphics[width=6cm]{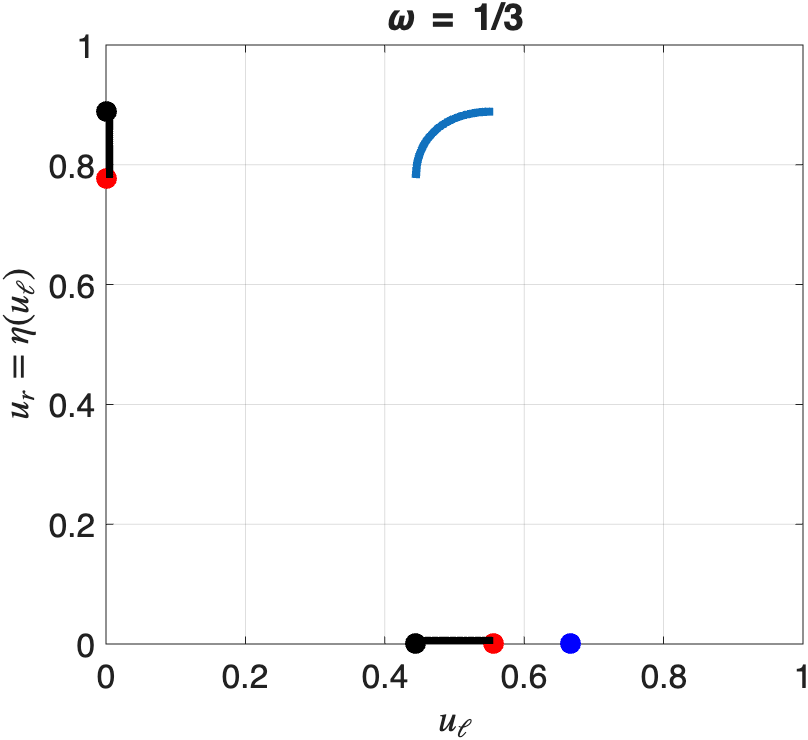}
    &
    \includegraphics[width=6cm]{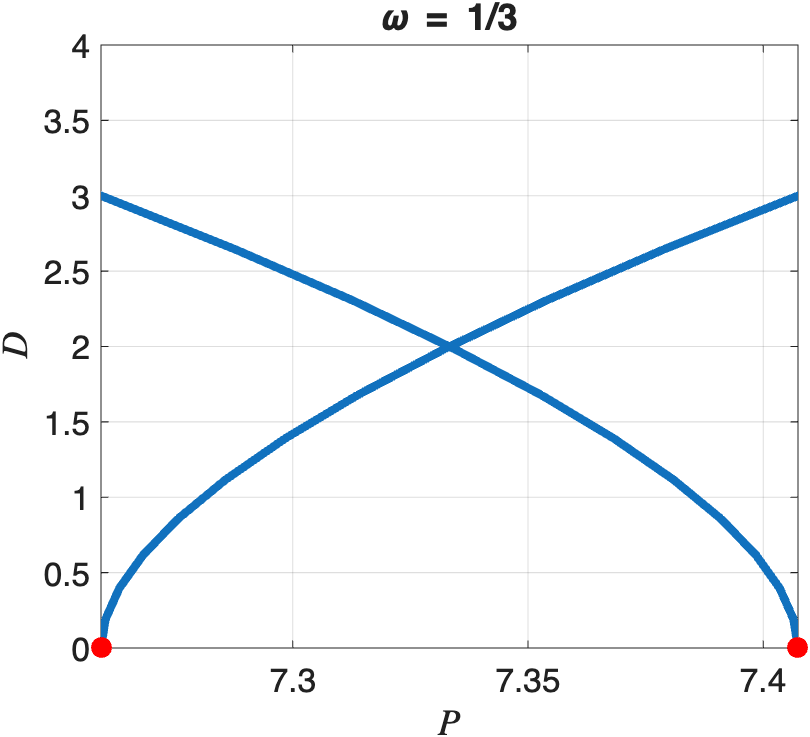}
\\
    \includegraphics[width=6cm]{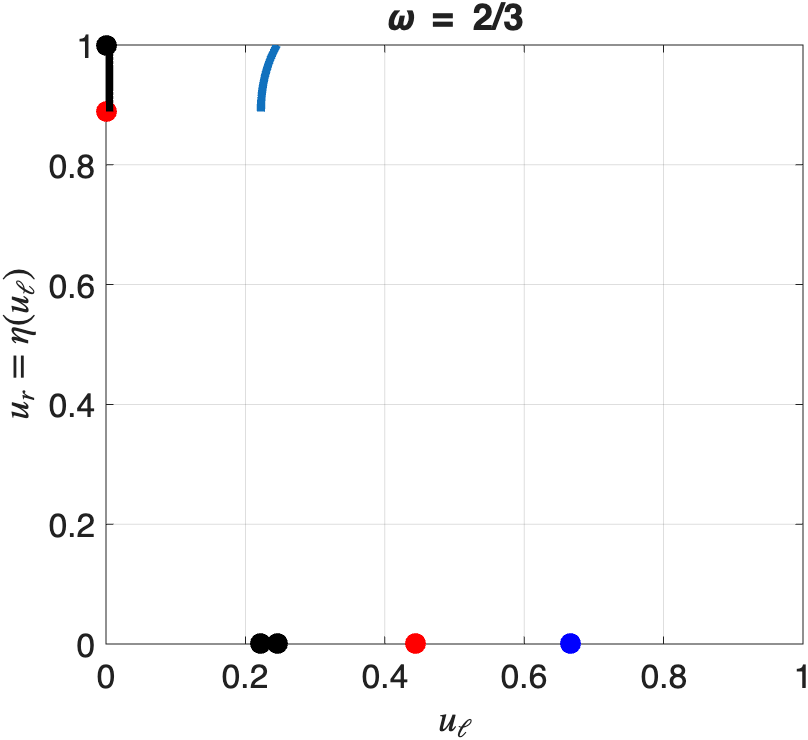}
    &
    \includegraphics[width=6cm]{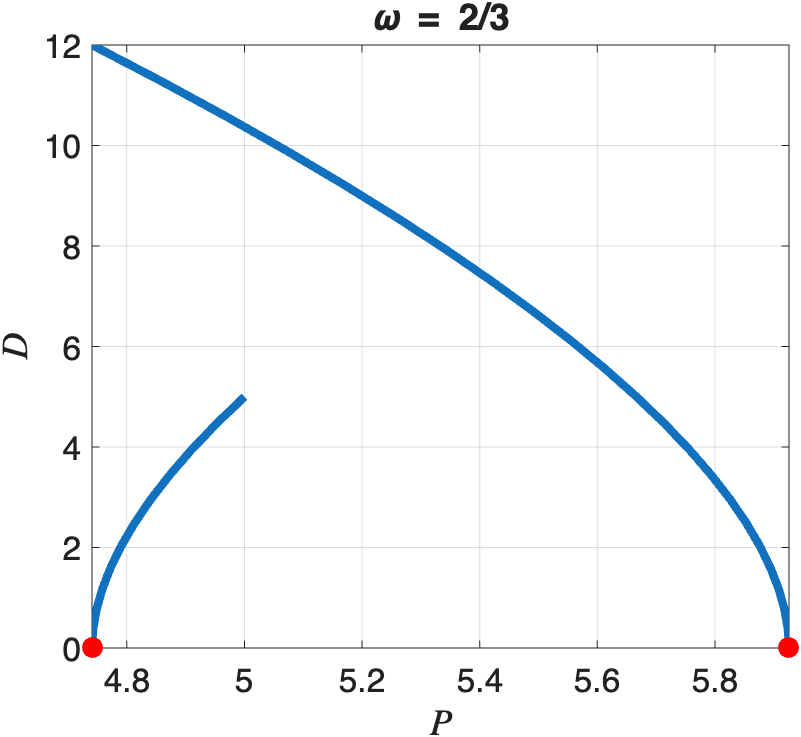}
    \end{tabular}
 \caption{Plots of the functions $\eta$ and $P$, $D$. On the left column the plot of $\eta$, on the right the plot of $D$ vs. $P$, where the red points denote $P(\beta)$ and $P(\alpha)$. In the top line the case $0<\omega<\frac12$, in the bottom line the case $\frac{1}{\sqrt{3}}<\omega<1$. Notation and values of $D_i$, $D_g$ are as in Figure \ref{f:triangle}. In the case $\frac12<\omega<\frac{1}{\sqrt{3}}$ the plot of $\eta$ is analogous to that in the second line, but $P$ and $D$ intersect.
 }
    \label{f:inter}
\end{figure}
In \cite{MTMWBH2} the authors showed that, when condition \eqref{e:D_cont} is satisfied, the jump $j(\phi_\ell):=\eta(\phi_\ell)-\phi_\ell$ attains its maximum. In the following proposition, we characterize where this maximum is achieved in general.

\begin{proposition} \label{p:jump}
Consider \eqref{e:coeffs D}-\eqref{e:coeffs g}, with $D_i>4D_g>0$ and \eqref{e:acb}. Then the jump of the shock wavefront  is maximum 

\begin{itemize}
\item if $\omega \in (0,1/\sqrt{3}]$ and $\varphi_\ell$ is determined by \eqref{e:D_cont};
\item if $\omega \in (1/\sqrt{3},1)$ and $\varphi_\ell=\varphi_\ell^*$ (where the speed is minimal).
\end{itemize}
\end{proposition}
\begin{proof}

From \eqref{eta_J}, the jump is given by 
\[
j(\phi_\ell)=1-\frac 32 \phi_{\ell}+\frac{\sqrt{\Delta_{\phi_{\ell}}}}{2}
\]
which implies 

\[
j'(\phi_\ell)=-\frac 32 +\frac{2-3\phi_{\ell}}{2 \sqrt{\Delta_{\phi_{\ell}}}}
\] 
and hence $j'(\phi_\ell)=0$ when $\phi_{\ell}=\frac 23 \pm \frac{\omega}{\sqrt 3}$. Note that  $\frac 23 + \frac{\omega}{\sqrt 3} \not \in \mathcal I_{\omega}$, whereas $\frac 23 - \frac{\omega}{\sqrt 3}\in \mathcal I_{\omega}$ only if $\omega \in (0, \frac{1}{\sqrt 3}]$ and, in this case, it is a maximum point for $j(\phi_\ell)$. In conclusion,  when $\omega \in ( 0, 1/\sqrt 3]$, from Proposition \ref{p:emm} we obtain that  the maximum of $j(\phi_{\ell})$ is achieved when condition \eqref{e:D_cont} is satisfied. For the remaining values of $\omega$ the maximum of $j(\phi_{\ell})$ is attained at the right endpoint of $\mathcal I_{\omega}$, i.e. when the corresponding speed is minimal (see Proposition \ref{prop:A}).
\end{proof}	

\begin{remark}
{
\rm According to Propositions \ref{prop:A}, \ref{p:emm}, and \ref{p:jump}, if $\omega \in (0, 1/\sqrt 3)$, the unique value $\phi_{\ell}$ for which condition \eqref{e:D_cont} is satisfied also maximizes the jump, although it does not minimize the corresponding speed. In contrast, when $\omega= 1/\sqrt 3$, the value $\phi_{\ell}=1/3$ simultaneously maximizes the jump, satisfies condition \eqref{e:D_cont}, and minimizes the corresponding speed (see Figure \ref{f:three}).
}
\end{remark}
\begin{figure}[htbp]
    \centering
    \begin{tabular}{cc}
    \includegraphics[width=7.2cm]{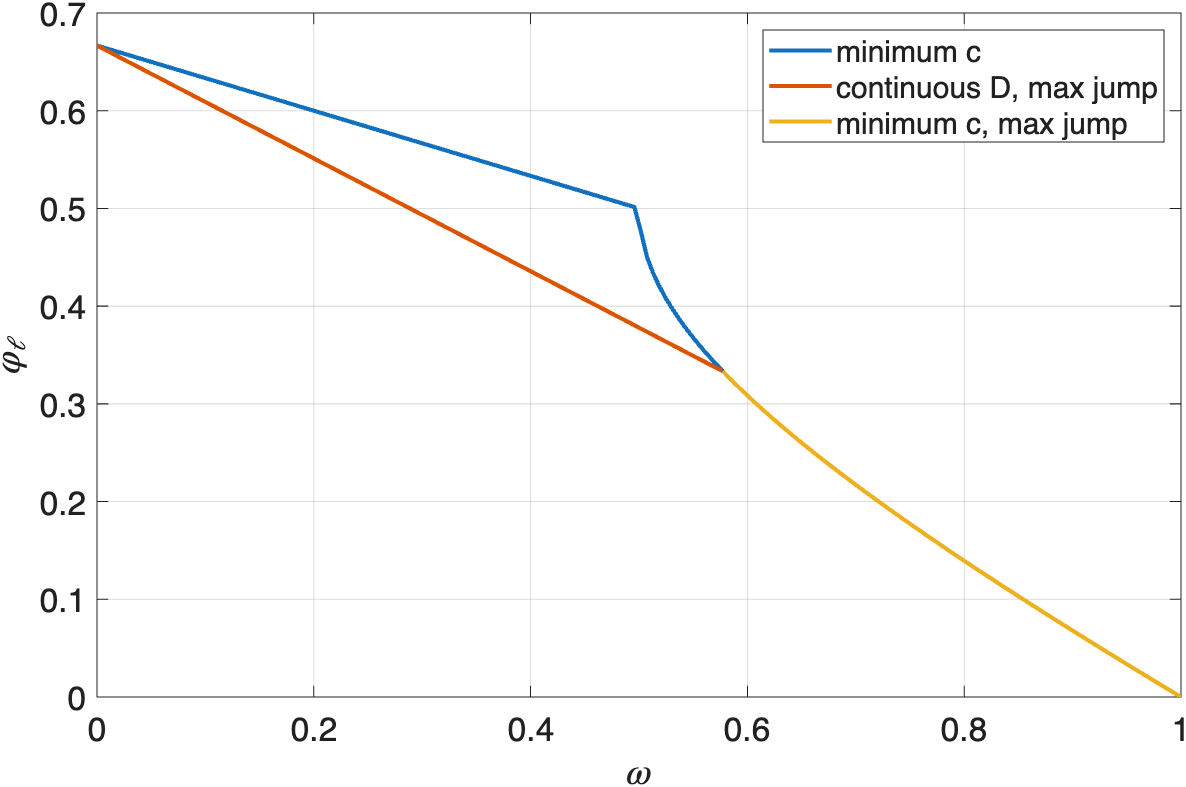}
        \end{tabular}
 \caption{The values of  $\phi_{\ell}$ for which condition \eqref{e:D_cont} is satisfied are shown in orange. The values of $\phi_{\ell}$ that minimize the corresponding speed are shown in blue and yellow.}
    \label{f:three}
    \end{figure}
Finally, we provide an estimate of the closed interval where the speeds vary. The result makes use of \eqref{e:c1c2} (see Proposition \ref{prop:c*_bounds}), derived in the general case.

\begin{proposition} \label{p:interval-c}
Consider \eqref{e:coeffs D}-\eqref{e:coeffs g}, with $D_i>4D_g>0$ and \eqref{e:acb}. Then,  $c^*(\phi_{\ell})$ satisfies the inequality
\begin{equation}\label{e:intervallo-c}
-2\sqrt{D_i(k_i-\lambda_i)} < c^*< 2\sqrt{\lambda_gD_g}.
\end{equation}
\end{proposition}
\begin{proof} From their definitions, it is straightforward to see that
\[
D(u)=3(D_i-D_g)(u-\alpha)(u-\beta), \quad g(u)=(k_i-\lambda_i+\lambda_g)u(1-u)(u-\gamma), \quad u\in [0,1]
\]
with $\gamma:=\frac{k_i-\lambda_i}{k_i-\lambda_i+\lambda_g}$, while  $\alpha$ and $\beta$ are defined in \eqref{e:alphabeta}. 
We use the estimate \eqref{e:c1c2}, concerning  the general case, and first derive an upper bound. Since 
\[
\frac{g(u)D(u)}{u-\beta}=3(D_i-D_g)(k_i-\lambda_i+\lambda_g)u(u-\alpha)(u-\gamma)(1-u),
\]
and since $u(1-u)$ is decreasing in $(\beta,1]$, by \eqref{e:alphabeta} it is easy to see that
\begin{align*}
\sup_{u\in (\beta, 1]} \frac{g(u)D(u)}{u-\beta}&\le 3(D_i-D_g)(k_i-\lambda_i+\lambda_g)\beta(1-\alpha)(1-\gamma)(1-\beta)
\\ 
&=3(D_i-D_g)(k_i-\lambda_i+\lambda_g)\beta\left(1-(\alpha+\beta)+\alpha\beta\right)(1-\gamma)
\\
&=(D_i-D_g)(k_i-\lambda_i+\lambda_g)\beta\frac{1-\omega^2}{3}(1-\gamma)= \beta\lambda_gD_g<\lambda_gD_g.
\end{align*}
Hence, $c^*<2\sqrt{\lambda_gD_g}$.
Similarly, since $u(1-u)<\frac14$ in $[0,1]$, then
\begin{align*}
\sup_{u\in [0, \alpha)} \frac{-g(u)D(u)}{\alpha-u}&\le \frac34(D_i-D_g)(k_i-\lambda_i+\lambda_g) \gamma =\frac34(D_i-D_g)(k_i-\lambda_i)< D_i(k_i-\lambda_i).
\end{align*}
\end{proof}

\section{Generalities on shock traveling waves }\label{e:MR}
\setcounter{equation}{0}
In this section we {\em only assume condition {\rm (R)}}. We first define regular and shock traveling-wave solutions; we also show some preliminary properties of their profiles. The definitions below, indeed, can be given under weaker assumptions; we refer to \cite{GK} for the regular case.  
We denote by $I\subseteq\R$ any open interval. By {\em jump point} of a function $f=f(\xi)$ defined in $I$ we mean a point $\xi_0\in I$ such that both side limits $f(\xi_0^\pm)$ {\em exist}, are {\em finite} and {\em different}; we denote $[f(\xi_0)]:=f(\xi_0^+)-f(\xi_0^-)$. 

\begin{definition}\label{d:regular-shock} 
Let $\varphi$ be a function defined in $I$ and valued in $[0,1]$, which is continuous in $I$ possibly apart from a finite number of jump points; let $c$ be a real constant. The function $u(x,t):=\varphi(x-ct)$, for $(x,t)$ with $\xi=x-ct \in I$, is a {\em traveling-wave} solution to equation \eqref{e:E} with speed $c$ and profile $\phi$ if it is a weak solution of equation \eqref{e:EP}, i.e., if
\begin{equation}\label{e:solution}
\int_I P(\phi)\psi''\, d\xi = c\int_I \phi\psi'\, d\xi -\int_I g(\phi)\psi\, d\xi, \qquad \psi\in C_c^\infty(I).
\end{equation}
In particular,  $u$ is a {\em regular traveling-wave} solution if $\phi$ is continuous, while it is a {\em shock traveling-wave}  solution  otherwise.
\end{definition}

When $\phi$ has a jump at some $\xi_s\in\R$ ($s$ for shock), we assume $\phi(\xi_s):=\phi(\xi_s^-)$ for definiteness.

\begin{remark}\label{r:aboutDef}
{\rm 
Let $\phi$ be a profile as in Definition \ref{d:regular-shock}.
By (R), the function   $P(\phi)$ is continuous in $I$ with at most a finite number of jump discontinuities; then $P(\phi) \in  L_{\rm loc}^1(I)$ and formula \eqref{e:solution} is well defined.
If $\phi$ is continuous, then Definition \ref{d:regular-shock}  is equivalent to \cite[Def. 2.2]{GK} (see the remarks following \cite[Def. 2.2]{GK}). 
}
\end{remark}

A (regular or shock) traveling wave (briefly, a TW) is {\em global} if $I=\R$; it is {\em classical} if $\varphi$ is differentiable, $P(\varphi)'$ is absolutely continuous and \eqref{e:EP} holds a.e.. A (regular or shock) {\em wavefront} is a global TW, whose profile $\phi$ is monotone, non-constant, and $\phi(\pm\infty)$ are zeroes of $g$. By monotone we mean that $\phi(\xi_1)\le\phi(\xi_2)$  if $\xi_1<\xi_2$. A (regular or shock) {\em semi-wavefront to} $1$ ({\em from} $1$) is a TW with $I=(a,\infty)$ ($I=(-\infty,a)$) for some $a\in\R$, whose profile $\phi$ is monotone, non-constant and $\phi(\xi)\to1$ as $\xi\to\infty$ ($\xi\to-\infty$). Semi-wavefronts to or from $0$ or $\gamma$ are defined analogously. As usual, we often address to profiles as traveling waves, with a slight abuse of terminology.

\smallskip
In the following two remarks, in \eqref{Ipm_Ijumps} and in \eqref{Ipm_0} we assume that  $\varphi$ is the profile of a traveling wave in $I$.

\begin{remark}\label{r:I0}
{
\rm
If $\varphi\equiv \varphi_0$ in a non-trivial open interval of $I$, then necessarily $g(\varphi_0)=0$. Otherwise, $\varphi$ cannot solve \eqref{e:EP}. Also, if $\phi$ is monotone, then for every $\sigma\in \varphi(I)$ with $g(\sigma)\ne 0$, there is at most a point $\xi_{\sigma}\in I$ such that $\phi(\xi_{\sigma})=\sigma$.
}
\end{remark}
As subsets of $I$, we define

\begin{equation}
\label{Ipm_Ijumps}
I_s := \{\text{jump points of $\varphi$ in $I$}\}, 
\qquad 
I_* := I \setminus I_s .
\end{equation}
Clearly, $I_*$ is composed by a finite number of open intervals and $\varphi$ is continuous in $I_*$. For our results, it is convenient to partition $I_*$ into the following subsets:
\begin{equation}
\label{Ipm_0}
I_{\pm} := \left\{ \xi \in I_* \,:\, \pm D\bigl(\varphi(\xi)\bigr) > 0 \right\},
\qquad
I_0 := \left\{ \xi \in I_* \,:\, D\bigl(\varphi(\xi)\bigr) = 0 \right\}.
\end{equation}

Possibly, either $I_\pm$ or $I_0$ can be empty.

\begin{remark}\label{r:I0_2}
{
\rm
Suppose ${\rm int}(I_0) \neq \varnothing$. Because $D$ has only isolated zeros, $\varphi$ must be constant in each connected component of ${\rm int}(I_0)$. For what observed in the first part of Remark \ref{r:I0}, such constant values must be common zeros of $D$ and $g$ (provided such zeros exist). We refer to \cite{BCM3} for the case when such a common zero ($\gamma$ in \cite{BCM3}) belongs to $(0,1)$.

Note,  $I_0$ is not necessarily closed, because its boundary $\partial I_0$ may intersects $I_{s}$. Moreover, in the case of a monotone $\varphi$, if ${\rm int}(I_0) \neq \varnothing$, then $\partial I_0$ consists of finitely many points.
}
\end{remark}

\smallskip
 We denote by $P(\phi)'$ the distributional derivative of the function $P(\phi)(\xi):=P\left(\phi(\xi)\right)$. In Proposition~\ref{p:Pconditions} we state some consequences of Definition~\ref{d:regular-shock}  on the regularity of $\varphi$, $P(\varphi)$ and $P(\varphi)'$. In particular, we show that $P(\phi)'$ is, in fact, a function which is continuous in $I_*$ and hence coincides there with the pointwise derivative of $P(\phi)$ in $I_*$. 

\begin{proposition}\label{p:Pconditions}
Assume {\rm (R)} and let $\phi$ be a traveling-wave solution of equation \eqref{e:E} defined on $I$ with speed $c$. Then, 
\begin{enumerate}[label=(\roman*)]
\item $P(\phi)$ is an absolutely continuous function in $I$ and $P(\phi) \in C^1(I_*)$;
\item $P(\phi)'+c\phi$ can be extended to $I_{s}$ by continuity;  
\item $\phi$ is of class $C^2$ and a classical solution of \eqref{e:EP} in ${\rm int}(I_*)$.
\end{enumerate}
\end{proposition}
\begin{proof} We prove each item in order.

\smallskip

{\em (i)} Let $G\in C(I)$ be an integral function of $g(\phi)$; we also have $G \in C^1(I_*)\cap L^1_{\rm loc}(I)$ because $g(\phi)$ is continuous in  $I_*$ and bounded in $I$. 
Since  $I_{s}$ consists of finitely many  points at most, then the distributional derivative of $P(\phi)'+c\phi-G$ is $0$ a.e.\! in $I$ (see  Definition~\ref{d:regular-shock}). Consequently $P(\phi)'+c\phi$ and $G$ differ a.e.\!  by a real constant \cite{Schwartz}, i.e.,
\begin{equation}\label{e:item1}
P(\phi)'=G-c\phi +k\quad \text{a.e. in }  I,
\end{equation}
with $k \in \mathbb{R}$. 
Therefore the distribution  $P(\phi)'$ is a function and, according to the properties of $\phi$ (see Definition \ref{d:regular-shock}), we have $P(\phi)'\in C(I_*)\cap L^1_{\rm loc}(I)$. Since $P(\phi)' \in  L^1_{\rm loc}(I)$, then $P(\phi)$ has an absolutely continuous representative (see e.g. \cite[Thm. 8.2]{Brezis}). As a consequence, $P(\phi)$ is absolutely continuous in $I$ because it has jump discontinuities at most (in $I_{s}$). Since, moreover,  $P(\phi)' \in C(I_*)$ we deduce $P(\phi) \in C^1(I_*)$.

\smallskip

{\em (ii)} Since  $P(\phi)'+c\phi=G \in L^1_{\rm loc}(I)$ by \eqref{e:item1},  the function   $P(\phi)'+c\phi$ has an absolutely  continuous representative in $I$ \cite[Thm. 8.2]{Brezis}. Since $P(\phi)'+c\phi$ continuous in $I_*$, see {\em (i)}, this implies that it is extendable to $I_{s}$ by continuity.

\smallskip

{\em (iii)} The function $\phi$ is constant in every connected component of ${\rm int}(I_0)$ (see Remark~\ref{r:I0_2}) and then $\phi \in C^2\left({\rm int}(I_0)\right)$. The set  $I_\pm$ is open since it does not include jump points of $\phi$; let $J $ be one of its connected components, i.e., an open interval contained in $I_{\pm}$ where $P'\left(\phi(\xi)\right)$ has constant non-zero sign. This implies, in particular, that the inverse of $P|_{\varphi(J)}$ is well-defined and $C^2$. By denoting $\Phi$ such an inverse, we have
\begin{equation}\label{e:P(phi)}
\phi(\xi)=\Phi\left(P(\phi(\xi))\right), \quad \xi \in J.
\end{equation}
This relation, and {\em (i)}, imply $\phi \in C^1(J)$. By \eqref{e:item1}, we infer that $P(\phi)' \in C^1(J)$, that is $P(\phi) \in C^2(J)$. By \eqref{e:P(phi)} again, we deduce $\phi \in C^2(J)$, that concludes the proof of {\em (iii)}.
\end{proof}

\begin{remark}\label{r:sidelimits}
{
\rm 
We comment on Proposition \ref{p:Pconditions}.

\smallskip

(i) The proof of Proposition~\ref{p:Pconditions}{\em (i)} only exploits the continuity of $g$ in $[0,1]$ but does not depend on the number  of its zeroes. Assume $P'>0$ a.e.\!\! in $[0,1]$, i.e., the diffusivity is positive but possibly degenerate. In this case, the continuity of $P(\phi)$ in Proposition~\ref{p:Pconditions}{\em (i)} rules out the possibility of shock traveling waves  for every merely continuous reaction term $g$. This applies in particular if $g$ is monostable, i.e., if $g$ continuous in $[0,1]$ with $g>0$ in $(0,1)$ and $g(0)=g(1)=0$. In this case the uniqueness of solutions which holds in the classical continuous framework \cite{GK} still holds in the wider class of solutions of Definition~\ref{d:regular-shock}.

\smallskip

(ii) Items {\em (i)} and {\em (ii)} of Proposition \ref{p:Pconditions} can be stated as 
\begin{equation}\label{e:llsk}
\left[P(\phi)\right]=0 \quad \text{and} \quad  \left[P(\phi)'+c\phi\right]=0 \ \mbox{ in } \ I.
\end{equation}
In fact, \eqref{e:llsk}$_1$ is equivalent to the continuity of $P(\phi)$ while \eqref{e:llsk}$_2$ holds because the function $P(\phi)'+c\phi$ is continuous in $I_*$ and extendable to $I_{s}$ by continuity.    

\smallskip

(iii) From Proposition~\ref{p:Pconditions}{\em (iii)} we have $P(\phi)\in C^1(I_*) $ and $P(\phi)$ has angular points in $I_{s}$. If $c=0$, then $P(\phi)'$ is well defined and continuous on $I$.    
}\end{remark}

In the following {\em we denote by $\xi_s$ a generic jump point of $\phi$ and by $0\le \phi_\ell<\phi_r \le 1$ the side limits of $\phi$ at $\xi_s$.} Conditions \eqref{e:llsk} are equivalent, at every jump point $\xi_s$ of $\phi$, to 
\begin{equation}\label{e:hypo}
P(\phi_r)-P(\phi_{\ell})=\int_{\phi_\ell}^{\phi_r}D(s)\,d s=0
\quad \mbox{ and } \quad
c=-\frac{\displaystyle P(\varphi)'(\xi_s^-)-P(\phi)'(\xi_s^+)}{\phi_r-\phi_\ell}.
\end{equation}

We now provide a regularity result about monotone profiles of both regular and shock traveling-waves.

\begin{lemma}\label{l:Dphine0} Assume {\rm (R)} and let $\phi$ be a traveling-wave solution of equation \eqref{e:E} defined on $I$ with speed $c$. If $\phi$ is monotone, then  $\phi'(\xi) \ne 0$ for every $\xi \in I_{\pm}$ such that $g(\phi(\xi))\neq 0$.
\end{lemma}
\begin{proof} Assume, by contradiction, the existence of $\xi_0 \in I_{\pm}$ with $g\left(\phi(\xi_0)\right)\ne \, 0$ and $\phi'(\xi_0)=0$. Since  $ I_{\pm}$ does not contain jump points of $\phi$, there is an open interval $J\subseteq I_{\pm}$ with $\xi_0 \in J $  where $D(\phi)$ is always different from zero. Hence $\phi\in C^2(J)$ by Proposition~\ref{p:Pconditions}{\em (iii)}, then also $P(\phi)\in C^2(J)$ and $\phi$ is a classical solution of \eqref{e:EP}; moreover, $P(\phi)'(\xi_0)=D\left(\phi(\xi_0)\right)\phi'(\xi_0)=0$ and $P(\phi)''(\xi_0)=-g\left(\phi(\xi_0)\right)\ne 0$ by \eqref{e:EP}. Hence the function $P(\phi)'=D(\phi)\phi'$ changes sign in a neighborhood of $\xi_0$; since $D$ has a fixed sign there, this implies that $\xi_0$ is a strict extremum point for $\phi$ in contradiction with the monotonicity of $\phi$. The proof is complete.
\end{proof}

It follows from Remarks \ref{r:I0} and \ref{r:I0_2} that for a {\em regular} traveling wave  $\phi$ it may occurr $\phi'=0$ a.e.\! in $I$. In this case $\phi$ is constant, i.e., $\phi \equiv \phi_0$, with $g(\phi_0)=0$.   The following result characterizes {\em shock}  profiles $\phi$ with $\phi'=0$ a.e.

\begin{proposition}\label{p:tratti}  Assume {\rm (R)}. Then $\phi $ is a shock traveling-wave solution in $I$ of \eqref{e:E} with $\phi'(\xi) = 0$ a.e.\! in $I$ if and only if there exist $-\infty\le a_0<a_1<\ldots<a_m\le \infty$, with $m>1$ and $I=(a_0, a_m)$, $\gamma_k\in [0,1]$ with $\gamma_k\ne \gamma_{k+1}$ for $k=1,\ldots,m-1$, such that 
\begin{equation}\label{e:ggPP}
g(\gamma_1)=\cdots =g(\gamma_m)=0
\quad \mbox{ and } \quad 
P(\gamma_1)=\cdots=P(\gamma_m).
\end{equation}
In this case, $\phi$ is a step function of the form
\begin{equation}\label{e:step}
\phi(\xi)
= \sum_{k=1}^{m} \gamma_k \, \chi_{k}(\xi), \quad  \xi\in I, 
\end{equation}
where \begin{equation*} \chi_k(\xi)=\begin{cases} 1 \ &\mbox{ if } \ \xi \in (a_{k-1},a_k]\cap I,
\\ 
0 \ &\mbox{otherwise}
\end{cases}
\end{equation*}
and  $c=0$.
\end{proposition}
\begin{proof}
Assume \eqref{e:ggPP}. The step function in \eqref{e:step}, defined on $I:=(a_0, a_m)$, makes $P(\phi)$ and  $g(\phi)$ constant on $I$. Moreover, since  $g(\phi)=0$ on $I$, condition \eqref{e:solution} reduces to $P(\gamma_1)\int_I \psi'' \, d\xi=c\int_I\phi \psi' \, d\xi$, which is satisfied for every $\psi\in C_c^\infty(I)$ iff $c=0$. Hence, $\phi$ is a shock traveling-wave of equation \eqref{e:E} with $I_{s}= \{a_1, ..., a_{m-1}\}$ and speed $c=0$.

Conversely, let $\phi$ be a shock traveling wave  of \eqref{e:E} with $\phi'=0$ a.e.\! in $I$. Then $I=(a_0,a_m)$ for some $-\infty\le a_0<a_m\le \infty$, $I_{s}=\{a_1, ..., a_{m-1}\}$ ($m>1$) and
\[
I_*=\displaystyle \bigcup_{k=1}^{m} (a_{k-1}, a_k).
\] 
The profile $\phi$ is necessarily constant in every interval  $(a_{k-1}, a_k]\cap I$, say $\phi:=\gamma_k$ with $\gamma_k \in [0,1]$ and $\gamma_k \ne \gamma_{k+1}$; then $\phi$ satisfies \eqref{e:step}. The step function $P(\phi)$ is continuous by Proposition \ref{p:Pconditions}{\em (i)}; then it is necessarily constant and $\eqref{e:ggPP}_2$ is satisfied. Furthermore, $P(\phi)'+c\phi=c\phi$ is continuous in $I_*=I\setminus I_s$ and can be continuously extended  to $I_s$ by Proposition \ref{p:Pconditions}{\em (ii)}; this implies $c=0$.  At last, the identities in $\eqref{e:ggPP}_1$ follow from \eqref{e:solution}.
\end{proof}


\section{Results on regular semi-wavefronts}\label{s:RSWFs}
\setcounter{equation}{0}
In this section, we assume conditions (R), (P) and (g). We combine results from \cite{BCM2, BCM3} to derive key tools concerning regular, monotone profiles, which will be instrumental in proving our main results on shock wavefronts, in Section \ref{s:WF}.

First, we introduce a change of variables which reduces \eqref{e:EP} to a singular first-order equation. As a consequence of Lemma \ref{l:Dphine0}, the profile of a monotone regular traveling wave is invertible in  $J := I_{\pm}\setminus \{\xi\in I \, : \, \phi(\xi)= 0, \gamma, 1\}$. The connected components of $\phi(J)$ are contained in the following sets
\begin{equation}\label{e:intervals}
(0, \alpha), \quad (\alpha, \gamma), \quad (\gamma, \beta),\quad (\beta, 1)
\end{equation} 
and the inverse function $\varphi^{-1}: \varphi(J) \to J$ is smooth. Then we can define
\begin{equation}\label{e:z}
z(\mu):=D(\mu)\phi'\left(\phi^{-1}(\mu)\right) \neq 0, \quad \mu \in \varphi(J).
\end{equation}
Since $D$ has constant sign on every interval in \eqref{e:intervals}, it follows that also $z$ has  constant sign on every connected component of $\phi(J)$ and $\sgn(z)=\sgn(D)$ if $\phi$ is increasing while $\sgn(z)=-\sgn(D)$ if $\phi$ is decreasing. By Proposition~\ref{p:Pconditions}, $\phi \in C^2(J)$, hence  $z \in C^1(\phi(J))$ and $z$ satisfies the first-order singular equation
\begin{equation}\label{eq:z}
  \dot z(\mu)=-c-\frac{D(\mu)g(\mu)}{z(\mu)}, \quad  \mu \in \phi(J).
\end{equation}
Therefore the function $z$ allows to reduce the study of the second-order equation \eqref{e:EP} to that of the first-order equation \eqref{eq:z}. 

For $E\subset \mathbb{R}$, let $\overline{E}$ denote the closure of $E$. Now, we show that $z$ can be continuously extended to $\overline{\phi(I)}$. 

 \begin{lemma}\label{l:extended} Let $\phi$ be the profile of a monotone regular traveling-wave solution and  $z$ the corresponding function as in  \eqref{e:z}. Then $z$ can be continuously extended to $\overline{\phi(I)}$. 
 \end{lemma}
 \begin{proof} In the following we assume $\phi$ decreasing; an increasing profile  $\phi$ is treated similarly.  
 By the continuity of $\phi$ we have $\overline{\phi(I)}=:[h,k]$ for some $h<k$. 

\smallskip

{\em (i)} Assume first that  $I= I_{\pm}$ and $\{\xi\in I \colon \phi(\xi)= 0, \gamma, 1\}=\emptyset$. In this case $\phi(I)$ is a subset of one of the intervals in \eqref{e:intervals} and the function $\frac{Dg}{-z}$ does not change sign in $\phi(I)$. Moreover  $z \in C^1(h,k)$ as noted above. So  it remains to show that $z$ can be extended by continuity to $[h,k]$. By integrating \eqref{eq:z} in $(\sigma, \eta) \subset \phi(I)$, we obtain 
 \begin{equation}\label{eq:zintegrale}
 z(\eta)-z(\sigma)=-c(\eta -\sigma) +\int_{\sigma}^{\eta}\frac{D(s)g(s)}{-z(s)} \, ds.
 \end{equation}
The limits
 \[
\lim_{\sigma \to h^+} \int_{\sigma}^{\eta}\frac{D(s)g(s)}{-z(s)} \, ds \quad \text{and} \quad  \lim_{\eta \to k^-} \int_{\sigma}^{\eta}\frac{D(s)g(s)}{-z(s)} \, ds 
\]
exist, being possibly $\pm \infty$, since both integrands have constant sign in $(h,k)$. We claim that the case $\pm \infty$ cannot occur. In fact, when
 \begin{equation}\label{eq:z1}
 \lim_{\eta \to k^-} \int_{\sigma}^{\eta}\frac{D(s)g(s)}{-z(s)} \, ds =\pm\infty,
 \end{equation}
then, by \eqref{eq:zintegrale}, the limit  $z(k^-)$ exists with the same value as the one in \eqref{eq:z1}. Hence, $\frac{Dg}{-z}$ is bounded in a left neighborhood of $k$, which contradicts \eqref{eq:z1}. Then the limit in  \eqref{eq:z1} is finite, and so also $z(k^-)$ is finite. An analogous reasoning applies to prove that $z(h^+)$ is finite, and then $z$ can be extended by continuity to $[h,k]=\overline{\phi(I)}$. 

\smallskip

{\em (ii)} Now let $\alpha \in \phi(I)$; then there exists a unique $\xi_{\alpha}\in I$ with  $\phi(\xi_{\alpha})=\alpha$ (see Remark~\ref{r:I0}). We claim that $\alpha \in {\rm int}\left(\phi(I)\right)$. In fact, by the monotonicity of $\phi$, if $\alpha \in \partial\left( \phi(I)\right)$ then either $I=(a, \xi_{\alpha}]$ with $a\in \mathbb{R}\cup\{-\infty\}$ or  $I=[\xi_{\alpha}, b)$ with $b\in \mathbb{R}\cup\{\infty\}$, in contradiction with $I$ open. This proves the claim. Hence there is $\delta>0$ such that $(\xi_{\alpha}-\delta, \xi_{\alpha})\cup (\xi_{\alpha}, \xi_{\alpha}+\delta) \subset  I_{\pm}\setminus \{\xi \, : \, \phi(\xi)= 0, \gamma, 1\}$. Consequently $z$ can be extended by continuity to each of the intervals $[\phi(\xi_{\alpha}+\delta), \alpha]$ and $[\alpha, \phi(\xi_{\alpha}-\delta)]$. Moreover
\begin{equation}\label{e:zinalfa}
z(\alpha^\pm)=\lim_{\xi \to \xi_{\alpha}^\pm}D\left(\phi(\xi)\right)\phi'(\xi)=\lim_{\xi \to \xi_{\alpha}^\pm}P(\phi)'(\xi).
\end{equation}
The two previous limits are equal by the continuity of $P(\phi)'$ in $\xi_{\alpha}$ (see Proposition~\ref{p:Pconditions}{\em (ii)}), and so $z$ is continuous in $\alpha$. The same reasoning applies to show the continuity of $z$ in $\beta$, if $\beta \in \phi(I)$, and in $\gamma$, if both $\gamma \in \phi(I)$ and $\phi^{-1}\{\gamma\}=\{\xi_{\gamma}\}$. 

\smallskip
{\em (iii)} It remains to consider the case when $\phi=\gamma$ on a non-trivial interval  $[\xi^1_{\gamma}, \xi^2_{\gamma}]\subset I$. Notice that $\phi'=0$, implying $P(\phi)'=0$,  in $(\xi^1_{\gamma}, \xi^2_{\gamma})$. Therefore, again by the continuity of $P(\phi)'$ in $[\xi^1_{\gamma}, \xi^2_{\gamma}]$ (see Proposition~\ref{p:Pconditions}), we obtain $z(\gamma^-)=z(\gamma^+)=0$ and $z$ is continuous in $\gamma$ also in this case. The proof is complete.
\end{proof}

\begin{remark}\label{r:z=0}
{\rm 
Let $\phi$ be the profile of a monotone, regular traveling wave. Assume $\alpha \in \phi(I)$.  By \eqref{e:zinalfa} and the sign conditions on $P'$ (see (P)) we deduce $z(\alpha)=0$. Similarly, $z(\beta)=0$ if $\beta \in \phi(I)$. On the contrary, if $\gamma \in \phi(I)$ we have $z(\gamma)=0$ when $\phi=\gamma$ on a non-trivial interval. }
\end{remark}

In the following lemma we show that the values $z(1)$ and $z(0)$ are prescribed
when $\phi$ is a semi-wavefront from $1$ or to $0$, respectively.  

\begin{lemma}\label{l:z1z0} If $\phi$ is a regular semi-wavefront from $1$ (to $0$), then  $z(1)=0$ ($z(0)=0$, respectively).
\end{lemma}
\begin{proof} Let $\phi$ be a regular semi-wavefront from $1$ defined on the interval $I$. Notice that $z(1)$ is well defined, by Lemma \ref{l:extended}, and it is  the limit of $z(\phi)$ when $\phi \to 1^-$. 

If there exists $\xi_1 \in I$ with $\phi(\xi)=1$ for $\xi \le \xi_1$ and $\phi(\xi)<1$ for $\xi > \xi_1$, then $z(1)=0$ because of the continuity of $P(\phi)'$ in $\xi_1$ (see Proposition~\ref{p:Pconditions}). 

Assume instead $\phi(\xi)<1$ for all $\xi \in I$. We claim that
\begin{equation}\label{e:minus_inf_limm}
\lim_{\xi \to -\infty} D\left(\phi(\xi)\right)\,\phi'(\xi) = 0.
\end{equation}
In fact, consider $(\zeta_1,\zeta_2)$ with $\phi(\zeta_2)>\beta$. Hence $\phi \in C^2(\zeta_1, \zeta_2)$ (see Proposition~\ref{p:Pconditions}) and integrating \eqref{e:EP} over $(\zeta_1,\zeta_2)$  we obtain:
\[
D\left(\phi(\zeta_1)\right)\phi'(\zeta_1)+c\phi(\zeta_1)
= D\left(\phi(\zeta_2)\right)\phi'(\zeta_2)+ c\phi(\zeta_2) -\int_{\zeta_1}^{\zeta_2} g\left(\phi(s)\right)\,ds.
\]
We know that $\phi(\zeta_1) \to 1$ as $\zeta_1 \to -\infty$. 
By the choice of $\zeta_2$ and the monotonicity of $\phi$ we also have that $g(\phi(s))>0$ on $(\zeta_1,\zeta_2)$. 
Consequently, the limit 
\begin{equation}\label{e:1limit}
\lim_{\zeta_1 \to -\infty} D(\phi(\zeta_1))\phi'(\zeta_1)
\end{equation}
exists. 
If \eqref{e:1limit} is different from zero, then $\phi'(\zeta_1) \to \ell\in [-\infty, 0)$ as $\zeta_1\to-\infty$; this contradicts the boundedness of $\phi$. Therefore the claim \eqref{e:minus_inf_limm} is true and it implies $z(1)=0$ by \eqref{e:z}. 

In the same way one proves that $z(0)=0$ if $\phi$ is a regular semi-wavefront to $0$.
\end{proof}
The following proposition concerns the existence of semi-wavefronts.

\begin{proposition}\label{p:twins}
Consider equation \eqref{e:EP}. Then, for every $c\in\R$ we have:
\begin{enumerate}[label=(\arabic*)]
\item There exists a unique regular semi-wavefront from $1$, connecting $1$ to $\beta$.

Denote by $\phi_c$ its profile and by $z_c$ the continuous extension to $[\beta,1]$ of the corresponding funtion defined in \eqref{e:z}.  If $c_1 < c_2$ then $z_{c_1}(\lambda) < z_{c_2}(\lambda)$ for every $\lambda \in (\beta, 1)$, and $z_{c_1}(\beta) \le z_{c_2}(\beta)$.

  \item There exists a unique regular semi-wavefront to $0$, connecting $\alpha$ to $0$.

Let $\psi_c$ be its profile and $z_c$ the continuous extension to $[0,\alpha]$ of the corresponding function defined in \eqref{e:z}. If $c_1 < c_2$ then $z_{c_1}(\mu) > z_{c_2}(\mu)$ for every $\mu \in (0,\alpha)$, and $z_{c_1}(\alpha)\ge z_{c_2}(\alpha)$.
\end{enumerate}
\end{proposition}

\begin{figure}[htb]
\begin{center}
\begin{tikzpicture}[>=stealth, scale=0.8]
\draw[->] (-1,0) --  (9,0) node[below]{$\phi$} coordinate (x axis);
\draw (0,0.1) node[above]{$0$};
\draw (2,0.1) node[above]{$\alpha$};
\draw (5,0.1) node[above]{$\sigma$};
\draw (6,0.1) node[above]{$\beta$};
\draw (8,0.1) node[above]{$1$};

\draw[ ->] (8,0) .. controls (7.4,1) and (6.6,1) .. (6.1,0.15) node[midway,above]{\footnotesize$c\in\R$}; 
\draw[->] (1.9,0.15) .. controls (1.4,1) and (0.7,1) .. (0.1,0.15)  node[midway,above]{\footnotesize$c\in\R$};

\draw[ ->] (8,0) .. controls (7,-1) and (6,-1) .. (5.1,-0.15) node[midway,below]{\footnotesize$c>0$}; 
\draw[->] (4.9,-0.15) .. controls (3,-1) and (2,-1) .. (0.1,-0.15)  node[midway,below]{\footnotesize$c<0$};

\filldraw[black] (0,0) circle (3pt);
\filldraw[black] (0,0) circle (3pt);
\filldraw[black] (5,0) circle (1pt);
\filldraw[black] (8,0) circle (3pt);
\draw (2,0) circle (3pt);
\draw (6,0) circle (3pt);

\end{tikzpicture}
\end{center}
\caption{\label{f:resume_2}{Pictorial representation of the results in Propositions~\ref{p:twins} and~\ref{p:2eta}. The arrows indicate the direction when $\xi$ increases.}}
\end{figure}
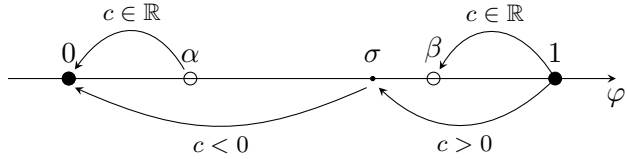
\noindent Given $\varphi \in [0,1]$, and a function $h$, defined in $[0,1]$,  we denote $\bar \varphi:=1-\varphi$, and
\begin{equation}
\label{barDtildeg}
\bar h(\varphi):=h(\bar\phi) \quad \mbox{ and } \quad \tilde h(\varphi):=-h(\bar\phi) \ \mbox{ for } \ \varphi \in [0,1].
\end{equation}

\begin{proof}[Proof of Proposition~\ref{p:twins}]
We divide the proof into two parts.

\medskip
\noindent\textit{Proof of (1).}
 Since $D > 0$ in $(\beta,1)$, $g > 0$ in $[\beta,1)$, and $D(\beta) =g(1) = 0$, the existence and uniqueness of semi-wavefronts from $1$ to $\beta$ follows from \cite[Proposition 3]{BCM2}.


As observed in \eqref{eq:z}, Lemmas~\ref{l:extended} and \ref{l:z1z0}, we have $z_c\in C[\beta, 1]$ and $z$ satisfies
\begin{equation}
\label{e:z_r}
\begin{cases}
\dot{z}(\varphi) = -c - \dfrac{D(\varphi)\,g(\varphi)}{z(\varphi)}, & \varphi \in (\beta,1), \\[4pt]
z(\varphi) < 0, & \varphi \in (\beta,1), \\[4pt]
z(1) = 0.
\end{cases}
\end{equation}
The uniqueness  of $z_c$ follows from \cite[Lemma 4.1(a)]{BCM3}, while $z_{c_1}(\lambda) < z_{c_2}(\lambda)$ for every $\lambda \in [\beta, 1)$ by \cite[Lemma 4.1(b)]{BCM3} if $c_1<c_2$. At last, we have $z_c(\beta) = 0$ for every $c$ greater than a finite (positive) threshold by \cite[Lemma 4.1(d) and Corollary 4.1]{BCM3}. This proves {\em (1)}.

\medskip
\noindent\textit{Proof of (2).}
Consider equation \eqref{e:EP} for $\psi$ replacing $\phi$ and make the change of variables \(\psi(\xi) := 1 - \varphi(-\xi)\). The equation for $\phi$ is
\begin{equation}\label{e:cambio_variabile}
\tilde{P}(\phi)''-c\phi'+\tilde{g}(\phi)=0.
\end{equation} 
Since \(\tilde{P}\,' > 0\), \(\tilde{g} > 0\) in \([\bar{\alpha},1)\) and \(\tilde{P} \,'(\bar{\alpha}) = \tilde{g}(1) = 0\), we deduce the existence and uniqueness of a unique profile $\phi$ from $1$ to  $\bar\alpha$ by \cite[Proposition 3]{BCM2}. The conclusion follows by inverting the transformation.

Let $z$ be the function defined as in \eqref{e:z} by using the profile $\psi$ and denote  $w(\phi):=z(1-\phi)$. The function $w$ is defined from the solution of \eqref{e:cambio_variabile} as in \eqref{e:z} and, thanks to \eqref{eq:z} and Lemma~\ref{l:z1z0}, it satisfies the boundary-value problem
\begin{equation}
\label{e:z_l}
\begin{cases}
\dot{w}(\varphi) = c - \dfrac{\bar{D}(\varphi)\, \tilde{g}(\varphi)}{w(\varphi)}, & \varphi \in (\bar{\alpha},1), \\
w(\varphi) < 0, & \varphi \in (\bar{\alpha},1), \\
w(1) = 0.
\end{cases}
\end{equation}
We conclude by applying to \eqref{e:z_l} the same discussion of problem \eqref{e:z_r}.
\end{proof}

\begin{proposition}\label{p:2eta}
Consider equation \eqref{e:EP} and let $\sigma \in (\alpha, \beta)$. Then:
\begin{enumerate}[label=(\arabic*)]
\item If there exists a regular semi-wavefront from $1$, connecting $1$ to $\sigma$, then $c > 0$.
\item If there exists a regular semi-wavefront to $0$, connecting $\sigma$ to $0$, then $c < 0$.
\end{enumerate}
\end{proposition}

\begin{proof}
We prove each statement separately.

\smallskip
\textit{(1)} Let $\phi$ be a semi-wavefront from $1$ to $\sigma$. By restriction, there exists a profile $\hat\phi: (-\infty, \xi_\beta) \to (\beta,1)$, with $\xi_\beta < \infty$. The function $z$ defined as in \eqref{e:z} satisfies $z(\beta) = 0$ by Remark \ref{r:z=0}; moreover,
we have $z(1) = 0$ by Lemma \ref{l:z1z0}, and $z < 0$ in $(\beta, 1)$. Since  $z$ satisfies \eqref{eq:z} we deduce $c \geq c^*$ for some $c^* > 0$  by \cite[Lem. 4.1, Cor. 4.1]{BCM3}.

\smallskip
\textit{(2)} Let $\phi$ be a semi-wavefront connecting $\sigma$ to $0$. Then $\psi(\xi) := 1 - \phi(-\xi)$ solves
\[
\begin{cases}
\tilde{P}\, ''(\psi)-c \psi' + g(\psi) = 0, \\
\psi(-\infty) = 1, \quad \psi(\bar{\xi}) = \bar{\sigma}.
\end{cases}
\]
The functions $\tilde{P}$, $\tilde{g}$ on $[\bar{\sigma}, 1]$ satisfy the same sign conditions as $P$, $g$ of part {\em (1)} on $[\sigma, 1]$,  but now the speed is opposite. Hence, the conclusion follows.
\end{proof}

\section{Results on shock wavefronts}
\label{s:WF}
\setcounter{equation}{0}

This section, where we always assume (R), (P) and (g), is organized as follows.  
In Subsection~\ref{sec:necessary}, we deduce a set of \emph{compatibility structural conditions} that are necessary for the existence of shock wavefronts (see Propositions~\ref{prop:necess} and~\ref{prop:necess2}).
Subsection~\ref{sec:suff} is devoted to the proof of Theorem~\ref{t:main}, and makes use of the analytical tools developed in the preceding sections. In Subsection \ref{ss:speed} we provide some estimates on the speed $c^*$. At last, in Subsection \ref{ss:concluding-remark} we comment on the proof of Theorem \ref{t:main} from the viewpoint of hyperbolic equations. 

\subsection{Compatibility structural conditions}
\label{sec:necessary}

We begin this section by characterizing profiles satisfying $\phi'=0$ a.e.

\begin{remark}
\label{rem:2jumps}
{\rm 
By Proposition \ref{p:tratti} we have that $\phi$ is a shock wavefront of \eqref{e:E}-\eqref{e:phi-limits} with  $\phi'=0$ a.e. in $\mathbb{R}$ if and only if \begin{itemize}
\item either $P(0)=P(1)$ and, for some $\xi_0 \in \mathbb{R}$,
\[
\phi(\xi)=\begin{cases}
1 & \text{for } \xi \in (-\infty,\xi_0], \\
0 & \text{for } \xi \in (\xi_0,\infty);
\end{cases}
\]

\item or $P(0)=P(\gamma)=P(1)$ and, for some $\xi_0^1, \xi_0^2 \in \mathbb{R}$ with $\xi_0^1 < \xi_0^2$,
\[
\phi(\xi)=\begin{cases}
1 & \text{for } \xi \in (-\infty, \xi_0^1], \\
\gamma & \text{for } \xi \in (\xi_0^1, \xi_0^2], \\
0 & \text{for } \xi \in (\xi_0^2,\infty).
\end{cases}
\]
\end{itemize}
As a consequence, $c=0$.}
\end{remark}
We now come back to the general case. For every profile $\varphi$ of a shock wavefront we can define the function $z$, as in \eqref{e:z}, in the interior of each connected component of $\varphi(\mathbb{R}) \setminus \{0, \alpha, \gamma, \beta, 1\}$. The condition $\eqref{e:hypo}_2$ can be stated as
\begin{equation}\label{e:cz}
c=-\frac{z(\phi_r)-z(\phi_\ell)}{\phi_r-\phi_\ell},
\end{equation}
where, as usual, $\phi_{\ell}$ and $\phi_r$ are the side limits of $\phi$ in a jump point. We refer to Subsection \ref{ss:concluding-remark} for a geometric interpretation of \eqref{e:cz}. Lemma~\ref{l:extended} still applies and implies that $z$ can be extended by continuity to $\overline{\varphi(\mathbb{R})}$. Lemma~\ref{l:z1z0} informs us that $z(0)=z(1)=0$. 

We now investigate the admissible jump locations and wave speeds of a shock wavefront.

\begin{lemma}\label{l:signC}
Assume that equation \eqref{e:E} admits a shock wavefront $\phi$ with speed $c$ satisfying \eqref{e:phi-limits}; let $\xi_s$ be a jump point and $\phi_\ell$, $\phi_r$ the corresponding side limits of $\phi$. Then we have:
\begin{itemize}
\item[(i)] $\phi_\ell \in [0, \beta)$; moreover, if $\phi_\ell \in (\alpha, \beta)$ then $c\ge 0$;

\item[(ii)] $\phi_r \in (\alpha, 1]$; moreover, if $\phi_r \in (\alpha, \beta)$ then $c\le 0$.
\end{itemize}
\end{lemma}

\begin{proof}
We only discuss {\em (i)} since the proof of {\em (ii)} is similar. 

If $\phi_\ell \in [\beta, 1]$ then $P(\varphi_r)>P(\varphi_\ell)$ and so $\eqref{e:hypo}_1$ cannot be satisfied that is, we get a contradiction with Proposition~\ref{p:Pconditions}(i); hence $\phi_\ell \in [0, \beta)$. If $\phi_\ell \in (\alpha, \beta)$, then necessarily $\phi_r \in (\beta, 1]$ again by $\eqref{e:hypo}_1$. Let $z$ be defined as in \eqref{e:z}  in $(\alpha,\varphi_\ell)\cup(\varphi_r,1)$. Then $z>0$ in $(\alpha,\varphi_\ell)$ and $z<0$ in $(\varphi_r,1)$; by Lemma~\ref{l:extended}, $z$ can be extended to $\phi_{\ell}$ and $\phi_r$, and so $z(\phi_{\ell})\ge 0$, $z(\phi_r)\le 0$. Hence, $c\ge0$ by \eqref{e:cz}.
\end{proof}

Lemma \ref{l:signC} does not rule out the cases $\varphi_\ell = \alpha$ or $\varphi_r =\beta$, which cannot occur at the same time because of $\eqref{e:hypo}_1$. 

Proposition~\ref{p:tratti} gives necessary and sufficient conditions about the existence of shock wavefronts with $\phi'=0$ a.e. In particular, in that case we have $P(1)=P(0)$ and $c=0$. 
The following proposition deals with the case when $\phi'\ne0$ a.e.

\begin{proposition}
\label{prop:necess}
Assume that \eqref{e:E} has a shock wavefront $\phi$ satisfying \eqref{e:phi-limits} with $\phi'\ne0$ a.e. Then
\begin{enumerate}[label=(\roman*)]
\item $P(1) > P(0)$;
\item $\varphi$ has a single jump point, and $(\varphi_\ell,\varphi_r) \in [0,\alpha]\times[0,\beta]$, with $(\varphi_\ell,\varphi_r) \neq (0,1)$.
\end{enumerate}
\end{proposition}

\begin{proof}
From $\eqref{e:llsk}_1$, the interval $(\phi_\ell, \phi_r)$ must contain at least one sign change of the function $P' = D$. Since $P' = D$ changes sign only twice (at the points $\alpha$ and $\beta$), the profile $\phi$ can have at most two jumps.

Suppose, by contradiction, that $\phi$ has two jump discontinuities, located at points $\xi_1>\xi_2$, corresponding to jumps from $\phi_r^{i}$ to $\phi_\ell^{i}$ for $i=1,2$, such that, since $\phi$ is decreasing,
\begin{equation*}
\label{e:two_jumps}
0 \le \phi_\ell^{1} < \alpha < \phi_r^{1} \le \phi_\ell^{2} < \beta < \phi_r^{2} \le 1.
\end{equation*}
Because $\phi_r^{1}, \phi_\ell^{2} \in (\alpha, \beta)$, by Lemma~\ref{l:signC} it follows $c = 0$. Hence $z(\varphi_\ell^1) = z(\varphi_r^1)$ and $z(\varphi_\ell^2) = z(\varphi_r^2)$ by \eqref{e:cz}. Since $z(\varphi_\ell^1), z(\varphi_r^2) \le 0$ and $z(\varphi_r^1), z(\varphi_\ell^2) \ge 0$, we must have
\[
z(\varphi_\ell^1) = z(\varphi_r^1) = z(\varphi_\ell^2) = z(\varphi_r^2) = 0.
\]
Since $\phi'\ne0$ a.e., it follows that $(\phi_\ell^1,\phi_r^1,\phi_\ell^2,\phi_r^2) \neq (0,\gamma,\gamma,1)$. Suppose $\phi_r^2 < 1$, since the other cases are treated similarly. We have
\[
z(\phi) < 0 
\qquad \text{and} \qquad 
\dot z(\phi) = -\frac{D(\phi) g(\phi)}{z(\phi)} > 0, 
\quad \text{for } \phi \in (\phi_r^2, 1),
\]
which contradicts $z(\phi_r^2) = 0$. Therefore, $\phi$ has only one jump, from $\varphi_r$ to $\varphi_\ell$.

From Lemma~\ref{l:signC}, we have $\phi_\ell \in [0, \beta)$ and $\phi_r \in (\alpha, 1]$. We now exclude the possibility that either $\phi_\ell$ or $\phi_r$ lies in the interval $(\alpha, \beta)$.

Assume by contradiction that $\phi_\ell \in (\alpha, \beta)$. Then we have $c \ge 0$ by Lemma~\ref{l:signC}. On the other hand, since $\phi_\ell \in (\alpha, \beta)$, Proposition~\ref{p:2eta}\textit{(2)} implies $c < 0$, a contradiction.
Similarly, suppose $\phi_r \in (\alpha, \beta)$. Then Lemma~\ref{l:signC} gives $c \le 0$, while Proposition~\ref{p:2eta}\textit{(1)} implies $c > 0$, again yielding a contradiction.
Hence, we deduce that $\phi_\ell \in [0, \alpha]$ and $\phi_r \in [\beta, 1]$. This proves item \emph{(ii)} because $\varphi'=0$ a.e. if $\varphi_\ell = 0$ and $\varphi_r = 1$.

We now prove {\em (i)}. Since $P$ strictly increases in $[0,\alpha]$ and in $[\beta,1]$, by $\eqref{e:llsk}_1$ we have
$0=P(\phi_r) - P(\phi_\ell) < P(1) - P(0)$.
\end{proof}

Because of Proposition~\ref{prop:necess}, when searching for a profile $\phi$ with $\phi'\ne0$ a.e., the notation $\xi_s$, $\phi_\ell$ and $\phi_r$ introduced above \eqref{e:hypo} refers to such a unique jump. 

\begin{definition}\label{d:admissible} We say that $(\varphi_\ell, \varphi_r) \in [0,\alpha] \times [\beta,1]$ is an \emph{admissible pair for a shock wavefront} if $\eqref{e:hypo}_1$ holds, that is, if $P(\varphi_\ell) = P(\varphi_r)$.  
\end{definition}
The next proposition shows that the set of admissible pairs is the graph of a regular, strictly increasing function defined on a closed interval $\mathcal I \subset [0,\alpha]$.

\begin{proposition}
\label{prop:necess2}
If $P(1)>P(0)$, every admissible pair $(\varphi_\ell,\varphi_r)$ for a shock wavefront is of the form $\left(\varphi_\ell, \eta(\varphi_\ell)\right)$, where $\varphi_\ell$ belongs to a closed, nonempty interval $\mathcal I \subset [0,\alpha]$, and $\eta: \mathcal I \to [\beta,1]$ is a function of class $C^0(\mathcal I) \cap C^1\left(\mathrm{int}(\mathcal I)\right)$ satisfying $\dot{\eta} > 0$ on $\mathrm{int}(\mathcal I)$.
\end{proposition}

\begin{proof}
A pair $(\varphi_\ell, \varphi_r)$ is admissible if and only if it belongs to the zero-level set of the function $\mathcal A$ defined by
\[
(\varphi, \psi) \in [0,\alpha]\times [\beta,1] \ \longmapsto\ 
\mathcal A(\varphi, \psi) := P(\psi) - P(\varphi).
\]
Clearly, $\mathcal A$ is continuous on $[0,\alpha]\times[\beta,1]$. Since 
$\mathcal A(0,1) = P(1) - P(0) > 0$ by hypothesis and 
$\mathcal A(\alpha,\beta) = P(\beta)-P(\alpha) < 0$, then 
the zero-level set of $\mathcal A$ is nonempty.  
We define
\[
\mathcal I := \left\{\, \varphi_\ell \in [0,\alpha] : \ \mathcal A(\varphi_\ell, \varphi_r) = 0.
\ \text{for some } \varphi_r \in [\beta,1] \,\right\}.
\]
Since the restrictions $P_1$ and $P_2$ of $P$ to $[0,\alpha]$ and $[\beta,1]$, respectively, are strictly increasing, then they are invertible and
\[
\mathcal A(\varphi_\ell,\varphi_r)=0 
\iff \varphi_\ell = P_1^{-1}\left(P_2(\varphi_r)\right)
\iff \varphi_r = P_2^{-1}(P_1\left(\varphi_\ell)\right).
\]

Suppose  $\mathcal A(\varphi_1, \psi_1) = \mathcal A (\varphi_2, \psi_2) = 0$ with $\varphi_2 > \varphi_1$.  
For every $\varphi \in (\varphi_1, \varphi_2)$, one has 
\[
\mathcal A(\varphi, \psi_1) = P(\psi_1) - P(\varphi) < 0 < P(\psi_2) - P(\varphi) = \mathcal A(\varphi, \psi_2).
\] 
By the Intermediate Value Theorem and the monotonicity of $P_2$, there exists a unique $\psi \in (\psi_1, \psi_2)$ such that $\mathcal A(\varphi,\psi)=0$.  
Hence $\mathcal I$ is connected, and therefore it is an interval contained in $[0,\alpha]$. Moreover, $\mathcal I$ is closed since it is the projection onto the $\varphi$-component of the zero level set of the continuous function $\mathcal A$. The function
\[
\eta(\varphi_\ell) := P_2^{-1}\left(P_1(\varphi_\ell)\right),
\quad \varphi_\ell \in \mathcal I,
\]
is well defined, and differentiable in ${\rm int}(\mathcal I)$. Moreover, for $\varphi_\ell \in {\rm int}(\mathcal{I})\subset(0,\alpha)$ we have $\eta(\varphi_\ell)\in (\beta,1)$ and so
\[
\dot \eta(\varphi_\ell) = \frac{D(\varphi_\ell)}{D\left(\eta(\varphi_\ell)\right)} > 0.
\] 
 
\end{proof}
It is clear from the proof of Proposition~\ref{prop:necess2} that the set of admissible pairs can also be described as $\left(\zeta(\phi_r),\phi_r\right)$, for $\phi_r$ in some closed nonempty interval $\mathcal{J}\subset[\beta,1]$ and $\zeta:\mathcal{J}\to[0,\alpha]$ a smooth function in $C^0(\mathcal J) \cap C^1\left(\mathrm{int}(\mathcal J)\right)$ satisfying $\dot{\zeta} > 0$ on $\mathrm{int}(\mathcal J)$.

\begin{remark}\label{r:struttura_I}
{\rm 
Assume $P(1)>P(0)$. The structure of the interval $\mathcal{I}$ depends on the relations between the values of $P$ at $0, \alpha, \beta$ and $1$. More precisely, $0 \in \mathcal{I}$ if and only if $P(0)\ge P(\beta)$, while $\alpha \in \mathcal{I}$ if and only if $P(\alpha)\le P(1)$. In particular, when both these conditions are satisfied,  then $\mathcal{I}=[0, \alpha]$.

\noindent Also $\eta(\mathcal{I})$ is an interval, by the continuity of the function $\eta$, and also its structure depends on the  relations between the values of $P$  at $0, \alpha, \beta$ and $1$. More precisely $\beta \in \eta(\mathcal{I})$ if and only if $P(\beta)\ge P(0)$, while $1 \in \eta(\mathcal{I})$ if and only if $P(\alpha)\ge P(1)$. In particular, when both these latter conditions are satisfied, then $\eta(\mathcal{I})=[\beta, 1]$.
}
\end{remark}
\subsection{The proof of Theorem \ref{t:main}}
\label{sec:suff}

The next result provides sufficient conditions under which, roughly speaking, a shock wavefront can be constructed by joining together two traveling waves; they can be either {\em constant} traveling waves or {\em regular} semi-wavefronts, one from $1$ and the other to $0$. 

Consider $c \in \mathbb R$, $\phi_r\in[\beta,1]$ and $\phi_\ell\in[0,\alpha]$. By Proposition~\ref{p:twins} there exist unique (up to shifts) regular semi-wavefronts $\hat\phi$ connecting $1$ with $\beta$ and $\check\phi$ joining $\alpha$ with $0$, with speed $c$. Then, for every $\xi_s \in \mathbb R$ we can suitably shift them and restrict at least one of their domains to define $\varphi:\mathbb R \to [0,1]$ by
\begin{equation}
\label{e:gluing}
\phi(\xi) := 
\begin{cases}
\hat\phi(\xi), & \xi \in (-\infty, \xi_s], \\
\check\phi(\xi), & \xi \in (\xi_s, \infty),
\end{cases}
\end{equation}
where, see Figure \ref{f:hatcheck}:
\begin{enumerate}[label=(\roman*)]
\item
if $\phi_r=1$ then $\hat\varphi\equiv 1$, otherwise $\hat\varphi(\xi_s^-)=\varphi_r$;

\item
if $\phi_\ell=0$ then $\check\varphi\equiv 0$, otherwise $\varphi_\ell:= \check\varphi(\xi_s^+)$. 
\end{enumerate}
%

\begin{figure}[htb]
\begin{center}
\begin{tikzpicture}[>=stealth, scale=0.65]

\draw[->] (-4,0) --  (6,0) node[below]{$\xi$} coordinate (x axis);
\draw[->] (0,0) -- (0,4.5) node[right]{$\phi$} coordinate (y axis);

\draw (-4,3.5) -- (6,3.5);
\draw[very thick] (-4,3.3) .. controls (-2,3.3) and (0,3.2).. (2,2.7) node[near start,below]{\footnotesize$\hat\phi$};
\draw[very thick] (2,1) .. controls (3,0.5) and (5,0.2) ..  (6,0.2) node[near end,above]{\footnotesize$\check\phi$};

\draw[very thick] (-4,3.5) -- (2,3.5) node[near start,above]{\footnotesize$\hat\phi$};
\draw[very thick] (2,0) -- (6,0) node[near end,below]{\footnotesize$\check\phi$};

\draw[dotted] (2,0) node[below]{\footnotesize$\xi_s$} --(2,3.5);
\draw[dotted] (0,2.7) node[left]{\footnotesize$\phi_r$}--(2,2.7);
\draw[dotted] (0,1) node[left]{\footnotesize$\phi_\ell$}--(2,1);

\draw (0,4.5) node[left]{$1$};

\draw(-0.1,2.3) node[left]{\footnotesize$\beta$}--(0.1,2.3);
\draw(-0.1,1.4) node[left]{\footnotesize$\alpha$}--(0.1,1.4);
\end{tikzpicture}

\end{center}
\caption{\label{f:hatcheck}{The traveling waves $\hat\phi$ and $\check\phi$.}}
\end{figure}
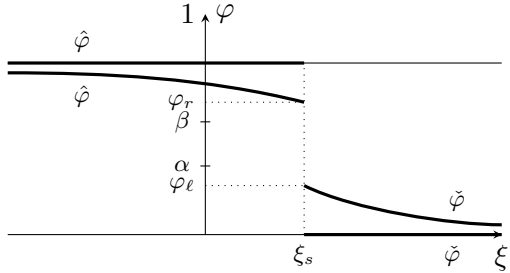

\begin{lemma}
\label{l:converse}
Assume $c \in \mathbb{R}$. Let $\phi$ be defined as in \eqref{e:gluing} for some $\xi_s$ in $\mathbb{R}$. If both conditions in \eqref{e:hypo} are satisfied, then $\varphi$ is the unique profile with speed~$c$ of a shock wavefront with a single jump from $\varphi_r$ to $\varphi_\ell$ located at $\xi_s$.
\end{lemma}

\begin{proof}
Let $\phi$ be defined as in \eqref{e:gluing} for some $c, \xi_s \in \mathbb{R}$. Assume  conditions  \eqref{e:hypo}. 

If  $(\varphi_\ell,\varphi_r)=(0,1)$, then \eqref{e:hypo}$_1$ implies $P(1)=P(0)$ and \eqref{e:hypo}$_2$ yields $c=0$. By Remark \ref{rem:2jumps} it follows that $\phi$ is the unique shock traveling wave solution with a single jump. 

Now consider $(\varphi_\ell, \varphi_r)\neq (0,1)$. Since the function $\varphi$ defined in \eqref{e:gluing} is globally defined, decreasing, and continuous except at $\xi_s$, then $\phi$ is a shock wavefront if it satisfies \eqref{e:solution}. Moreover, since we already know that $\phi$ is a solution in $(-\infty,\xi_s)\cup(\xi_s,\infty)$, it suffices to show that $\phi$ also solves \eqref{e:solution} in a neighborhood $(a,b)$ of $\xi_s$, with $a,b\in \mathbb{R}$. We may further assume that $(a,b)$ is small enough not to contain the points 
$\xi_1 := \inf\{\xi:\, \varphi(\xi)<1\}$ and $\xi_2 := \sup\{ \xi:\, \varphi(\xi)>0\}$,
whenever they are finite. As a conclusion, we have $\varphi \in C^2\left((a,b)\setminus\{\xi_s\}\right)$, by Proposition~\ref{p:Pconditions}{\em (iii)}. %

Let $\psi \in C^\infty_c(\mathbb R)$, with ${\rm supp}(\psi)\subset (a,b)$, with $a$ and $b$ as above. We need to prove 
 \begin{equation}
 \label{e:intCa}
 \int_a^b P(\varphi)\psi''-c\varphi \psi' +g(\varphi)\psi\,d\xi = 0.
 \end{equation}
Denote $\mathcal L:= P(\varphi)\, \psi'' - c \varphi \,\psi' +g(\varphi)\,\psi$.
In every interval $(c,d)\subset(a,b)$ such that $\varphi \in C^2(c,d)\cap C^0[c,d]$ and $P(\varphi)' \in C^0[c,d]$  we have
\begin{align}
\int_c^d \mathcal L\,d\xi &= \int_c^d P(\varphi)\, \psi'' - c \varphi \,\psi' +g(\varphi)\,\psi\,d\xi
\nonumber
\\
&= P(\varphi)\,\psi'|^d_c + \int_c^d  -\left(P(\varphi)' +c\varphi\right)\psi' +g(\varphi)\,\psi\,d\xi 
\nonumber
\\
&= P(\varphi)\,\psi'|^d_c -\left(P(\varphi)'+c\varphi\right)\psi|^d_c+\int_c^d  \left(P(\varphi)''+c\varphi' +g(\varphi)\right)\psi\,d\xi
\nonumber
\\
&= P(\varphi)\,\psi'|^d_c -\left(P(\varphi)'+c\varphi\right)\psi|^d_c,
\label{e:heks}
\end{align}
since $P(\varphi)''+c\varphi' +g(\varphi)=0$ in $(c,d)$. By Lemma \ref{l:extended}, $P(\varphi)'$ is continuously extendable to the intervals $[a,\xi_s]$ and $[\xi_s,b]$; therefore we can apply \eqref{e:heks} to both of them. Then \eqref{e:intCa} holds if and only if
\begin{equation*}
\begin{aligned}
0&=\int_a^b \mathcal L \, d\xi  = \int_a^{\xi_s} \mathcal L \,d\xi+ \int_{\xi_s}^b \mathcal L\,d\xi
\\
&= P(\hat\varphi)\psi'|_a^{\xi_s} + P(\check\varphi)\psi'|_{\xi_s}^b
-\left(P(\hat\varphi)'+c\hat\varphi\right)\psi|_a^{\xi_s} - \left(P(\check\varphi)'+c\check\varphi\right)\psi|_{\xi_s}^b
\\
&=[P(\varphi)]\psi'(\xi_s)-\left[P(\varphi)'+c\varphi\right]\psi(\xi_s).
\end{aligned}
\end{equation*}
Since $\psi$ is arbitrary, \eqref{e:intCa} holds if and only if  \eqref{e:llsk} are satisfied. These conditions are equivalent to those in  \eqref{e:hypo}. Therefore, $\phi$ is a shock wavefront. The uniqueness (up to shifts) of $\phi$ then follows from Proposition~\ref{p:twins}.
\end{proof}

We have now all the ingredients to prove Theorem~\ref{t:main}.

\begin{proofof}{Theorem \ref{t:main}}
We have three cases.

If $P(1)<P(0)$, then no shock wavefront $\phi$ with $\phi'\ne0$ a.e. exist because of Proposition~\ref{prop:necess}. Shock wavefronts with $\phi'=0$ a.e. cannot exist as well by Remark~\ref{rem:2jumps}.

If $P(1)=P(0)$, then only piecewise constant solutions are possible:  
since, by the very definition of wavefront, the profile must be decreasing, the only admissible piecewise-constant shock wavefronts are those described in Remark \ref{rem:2jumps}.  This completes the proof of items (2) and (3) in Theorem \ref{t:main}.

If $P(1)>P(0)$, then, by Propositions~\ref{prop:necess} and~\ref{prop:necess2} and Remark~\ref{rem:2jumps}, any shock wavefront must correspond to a profile with a single discontinuity point, associated with an admissible pair $(\varphi_\ell, \varphi_r)=\left(\varphi_\ell, \eta(\varphi_\ell)\right)$, where $\varphi_\ell \in \mathcal I$.  
Thus, to complete the proof of Item (1), it remains to show that for every admissible pair there exists one and only one speed $c^*=c^*(\varphi_\ell)$, that it is continuous and strictly decreasing on $\mathcal I$. 

\medskip
\emph{Existence of $c^*(\varphi_\ell)$.}

\noindent Let $\phi$ be defined as in \eqref{e:gluing} with $(\phi_{\ell}, \phi_r)$ an admissible pair (see Definition ~\ref{d:admissible}). Consider first the case $\varphi_\ell \neq 0$ and  $\eta(\varphi_\ell)=\phi_r\neq 1$.  
Since condition $\eqref{e:hypo}_1$ holds by definition of admissible pair, if condition $\eqref{e:hypo}_2$ is also satisfied, then, due to Lemma~\ref{l:converse}, $\varphi$ is the profile of a shock wavefront with speed $c$ and a single jump from $\varphi_r $ to $\varphi_\ell$.

Let $z$ be the continuous extension to $\overline{\varphi(\mathbb R)}$ of the function defined in \eqref{e:z} (see Lemma~\ref{l:extended}), and denote it by $z_c$ to emphasize its dependence on $c$.  
By \eqref{e:cz}, to establish the result it suffices to show that for every $\varphi_\ell \in \mathcal I$ there is a unique $c^*=c^*(\phi_\ell)$ such that
\begin{equation}
\label{e:constraint_z}
c^* = -\frac{z_{c^*}\left(\eta(\varphi_\ell)\right) - z_{c^*}(\varphi_\ell)}{\eta(\varphi_\ell) - \varphi_\ell}.
\end{equation}

For $c \in \mathbb{R}$ and $\varphi_\ell \in \mathcal I$, define
\begin{equation}\label{e:mathcalF}
\mathcal F(c,\varphi_\ell)
:= z_c\left(\eta(\varphi_\ell)\right) - z_c(\varphi_\ell)
   + c\left(\eta(\varphi_\ell) - \varphi_\ell\right).
\end{equation}
Clearly, $\mathcal F(c^*,\varphi_\ell) = 0$ if and only if \eqref{e:constraint_z} holds.
For any fixed $\varphi_\ell \in \mathcal I$, we have that $z_c(\eta(\varphi_\ell))$ is increasing and $z_c(\varphi_\ell)$ is decreasing with respect to  $c$ because of Proposition~\ref{p:twins}.  
Hence $\mathcal F(c,\varphi_\ell)$ is strictly increasing as a function of $c$, because of $\eta(\varphi_\ell)>\varphi_\ell$.
Moreover, by its definition in \eqref{e:z},  $z_c<0$ both in $(0, \phi_{\ell})$ and in $(\phi_r, 1)$; hence, for $c \in \mathbb{R}$ and $\varphi_\ell \in \mathcal I$ we have
\[
z_c(\eta(\varphi_\ell)) + c(\eta(\varphi_\ell) - \varphi_\ell)
\le \mathcal F(c, \varphi_\ell)
\le -z_c(\varphi_\ell) + c(\eta(\varphi_\ell) - \varphi_\ell).
\]
Observe that
\[
\lim_{c \to \infty} 
\big[z_c(\eta(\varphi_\ell)) + c(\eta(\varphi_\ell) - \varphi_\ell)\big] = \infty,
\]
since $z_c(\eta(\varphi_\ell)) \le 0$ is increasing in $c$, and
\[
\lim_{c \to -\infty} 
\big[-z_c(\varphi_\ell) + c(\eta(\varphi_\ell) - \varphi_\ell)\big]
= \lim_{c \to \infty} \big[-z_{-c}(\varphi_\ell) - c(\eta(\varphi_\ell) - \varphi_\ell)\big]
= -\infty,
\]
because $-z_{-c}(\varphi_\ell) \ge 0$ is decreasing in $c$.  
Thus, $\mathcal F$ takes on both positive and negative values.

The function $\mathcal F(c,\varphi_\ell)$ is continuous with respect to $\varphi_\ell\in\mathcal{I}$, since both $\eta$ and $z_c$ are continuous in $\mathcal{I}$ and $\overline{\phi(\R)}$, respectively (see Proposition~\ref{prop:necess2} and Lemma~\ref{l:extended}, respectively).  

The continuity of $\mathcal{F}$ with respect to $c$ follows from an application of  \cite[Lemma 3.3]{BCM1}, as we now show.  Fix $c_0 \in \mathbb{R}$, and consider a sequence $c_n \to c_0$. Define $z_n := z_{c_n}$ and $z_0 := z_{c_0}$ on $[0,\alpha] \cup [\beta,1]$. We first deal with the case where $\{c_n\}_n$ is increasing, and denote $v_n := z_n|_{[\beta,1]}$ and $v_0 := z_0|_{[\beta,1]}$. Recall that in $(\beta,1)$, $v_n$ satisfies \eqref{eq:z} with $c = c_n$, while $v_0$ satisfies \eqref{eq:z} with $c = c_0$. By Proposition~\ref{p:twins}\emph{(1)}, the sequence $\{v_n\}_n$ is increasing and $v_n \le v_0$.  Since $D>0$ and $g>0$ in $(\beta,1)$, we can apply \cite[Lemma 3.3]{BCM1}, which implies that $v_n \to v_0$ uniformly on $[\beta,1]$.  Analogously, define $w_n(\varphi) := z_n(1-\varphi)$ for $\varphi \in [0,\alpha]$, and set $w_0(\varphi) := z_0(1-\varphi)$. As observed in the proof of Proposition~\ref{p:twins}\emph{(2)}, $w_n$ satisfies \eqref{e:z_l}$_1$ with $\bar D, \tilde g >0$ as in \eqref{barDtildeg} defined in $(\bar\alpha,1)$. Moreover, by Proposition~\ref{p:twins}\emph{(2)}, the sequence $\{w_n\}_n$ is increasing and $w_n \le w_0$. Hence \cite[Lemma 3.3]{BCM1} applies again, giving $w_n \to w_0$ uniformly on $[\bar\alpha,1]$ and then  $z_n|_{[0, \alpha]} \to z_0|_{[0, \alpha]}$ uniformly in $[0, \alpha]$.  By similar arguments, the case of a decreasing sequence $\{c_n\}_n$ can be treated in the same way using \cite[Lemma 3.3]{BCM1}. We  have proved that $z_n \to z_0$ uniformly on $[0,\alpha] \cup [\beta,1]$, which in particular implies that $\mathcal F(c_n, \varphi_\ell) \to \mathcal F(c_0, \varphi_\ell)$ for every $\varphi_\ell$, as desired. This shows the continuity of $\mathcal{F}$ with respect to $c$.

By the Intermediate Value Theorem and the strict monotonicity of $\mathcal{F}$ with respect to $c$, we deduce that for every $\varphi_\ell \in \mathcal I$ there exists a unique $c^*(\varphi_\ell)$ such that $\mathcal F\left(c^*(\varphi_\ell),\varphi_\ell\right)=0$.

Finally, if $\varphi_\ell = 0$ (or $\eta(\varphi_\ell)=1$), all arguments still apply by considering $\hat{\varphi}\equiv 0$ (or $\hat{\varphi}\equiv 1$); in this case we have $z_c(\varphi_\ell)=0$ (or $z_c(\eta(\varphi_\ell))=0$) for every $c\in\mathbb{R}$.

\medskip
\emph{Continuity of $c^*(\phi_{\ell})$.}

\noindent Fix $\varphi_0 \in \mathcal I$ and set $c_0 = c^*(\varphi_0)$.  
Since $\mathcal F(\cdot,\varphi_0)$ is strictly increasing, for every $\varepsilon > 0$ we have
\[
\mathcal F(c_0 - \varepsilon, \varphi_0) < 0 < \mathcal F(c_0 + \varepsilon, \varphi_0).
\]
By the continuity of $\mathcal F(c_0 - \varepsilon, \cdot)$ and $\mathcal F(c_0 + \varepsilon, \cdot)$, there exists a neighborhood $U_0=U_0(c_0)$ of $\varphi_0$ such that
\[
\mathcal F(c_0 - \varepsilon, \varphi) < 0 < \mathcal F(c_0 + \varepsilon, \varphi),
\quad \text{for all } \varphi \in U_0.
\]
Consequently, we have $|c^*(\varphi) - c_0| < \varepsilon$ 
for all  $\varphi \in U_0$. Since $\varepsilon>0$ is arbitrary, this proves that $c^*$ is continuous at $\varphi_0$.

\medskip
\emph{Monotonicity of $c^*(\phi_{\ell})$.}

\noindent We compute, for $\phi_{\ell}\in{\rm int}(\mathcal{I})$, 
\begin{equation}\label{e:ift}
\begin{aligned}
\partial_{\phi_{\ell}} \mathcal{F}(c, \phi_\ell)
&= \dot{z_c}(\eta(\phi_\ell))\, \dot{\eta}(\phi_\ell) - \dot{z_c}(\phi_\ell) + c\dot{\eta}(\phi_\ell) - c \\
&= -\frac{(Dg)(\eta(\phi_\ell))}{z_c(\eta(\phi_\ell))}\, \dot{\eta}(\phi_\ell)
   + \frac{(Dg)(\phi_\ell)}{z_c(\phi_\ell)}.
\end{aligned}
\end{equation}
Since, by Proposition~\ref{prop:necess2}, we have
\[
(Dg)(\eta(\phi_\ell)) > 0,\quad \dot{\eta}(\phi_\ell) > 0,\quad (Dg)(\phi_\ell) < 0,
\quad \text{for } \phi_\ell \in {\rm int}(\mathcal I),
\]
we conclude that $\partial_{\phi_\ell} \mathcal{F}(c, \phi_\ell) > 0$ for $\phi_\ell \in {\rm int}(\mathcal I)$.
Hence, if $\varphi_1 < \varphi_2$ and $c_1 := c^*(\varphi_1)$, $c_2 := c^*(\varphi_2)$, we have
$\mathcal F(c_2,\varphi_1) < \mathcal F(c_2,\varphi_2) = 0$.
Since $c \mapsto \mathcal F(c,\varphi_1)$ is strictly increasing, it follows that $c_1 < c_2$.  
Therefore, $c^*(\varphi_\ell)$ is strictly decreasing in $\mathcal{I}$.
\end{proofof}

\subsection{Estimates on the speed}\label{ss:speed}
In this subsection we provide upper and lower estimates for the speed $c^*(\phi_\ell)$ in Theorem \ref{t:main}. 
Notice that if $P(\beta) \ge P(0)$ and $P(\alpha) \le P(1)$, then $P(1)>P(0)$.

\begin{proposition}
\label{prop:c*_bounds}
Assume $P(\beta) \ge P(0)$ and $P(\alpha) \le P(1)$.
Then, for every $\varphi_\ell \in \mathcal I$, we have
\begin{equation}
\label{e:c1c2}
-2\sqrt{\sup_{[0, \alpha)}\frac{1}{\alpha - \phi}\int_{\phi}^{\alpha}\frac{-g(\sigma)D(\sigma)}{\alpha-\sigma}\, d\sigma} < c^*(\varphi_\ell) < 2\sqrt{\sup_{(\beta,1]}\frac{1}{\phi -\beta}\int_{\beta}^{\phi}\frac{g(\sigma)D(\sigma)}{\sigma-\beta}\, d\sigma}.
\end{equation}
\end{proposition}

\begin{proof} The existence of real numbers $c_{-}<c_{+}$ such that $c^*(\varphi_\ell) \in [c_-,c_+]$, for every $\varphi_\ell \in \mathcal I$, follows from the continuity of $c^*$ on the closed interval $\mathcal I$, as stated in Theorem \ref{t:main}. In the following we look for explicit expressions of these lower and upper bounds. 

\medskip
{\em Definition of $c_{+}$.} 
Since $P(\beta)\ge P(0)$, we have $\beta \in \eta(\mathcal I)$ (see Remark \ref{r:struttura_I}). For every $c\in \mathbb{R}$ consider the unique (unless of shifts)  regular semi-wavefront $\hat \phi_c$ from $1$ to $\beta$ whose existence is given by Proposition \ref{p:twins}{\em (1)}, and let $z_c$ be the corresponding function associated with $\hat \phi_c$ as defined in \eqref{e:z}.  By  \cite[Proposition 3 \text{and} (2.11)]{BCM2} there is $c_{1\beta}^*\ge 0$ such that $z_c(\beta)=0$ if and only if $c \ge c^*_{1\beta}$. More precisely, $c_{1\beta}^*>0$; in fact, if $c_{1\beta}^*=0$, from \eqref{eq:z} we would obtain $\dot z_{c_{1\beta}^*}>0$ in $(\beta, 1)$ which contradicts the fact that $z_{c_{1\beta}^*}(\beta)<0$ in $(\beta, 1)$. Let $\zeta=\eta^{-1}:\eta(\mathcal I)\to \mathcal I$. Recalling the definition of $\mathcal{F}$ in \eqref{e:mathcalF}, we deduce
\[
\begin{aligned}
\mathcal F\left(c, \zeta(\beta)\right)&=z_c(\beta)-z_c\left(\zeta(\beta)\right)+c\left(\beta-\zeta(\beta)\right)
\\
&=
-z_c\left(\zeta(\beta)\right)+c\left(\beta-\zeta(\beta)\right) >0 
\quad \text{for } c\ge c^*_{1\beta},
\end{aligned}
\]
that implies $c^*\left(\zeta(\beta)\right)<c^*_{1\beta}$.  Since $\varphi_r\mapsto c^*\left(\zeta(\varphi_r)\right)$ is strictly decreasing by Theorem \ref{t:main} and Proposition~\ref{prop:necess2}, it follows
\[
c^*(\varphi_\ell) \le c^*\left(\zeta(\beta)\right) < c^*_{1\beta}.
\]
Thus we set $c_{+}:=c^*_{1\beta}$. Moreover, by making use of \cite[Corollary 4.1 and Remark 4.1]{BCM3}, we have
\begin{equation}
\label{e:est_c2}
2\sqrt{g(\beta)D'(\beta)} \le c_{1\beta}^*\le 2\sqrt{\sup_{(\beta,1]}\frac{1}{\phi -\beta}\int_{\beta}^{\phi}\frac{g(\sigma)D(\sigma)}{\sigma-\beta}\, d\sigma}.
\end{equation}

\medskip
{\em Definition of $c_{-}$.} 
The assumption $P(1)\ge P(\alpha)$ implies  $\alpha \in \mathcal I$. Define $\bar z(\varphi):=z_c(1-\varphi)$, so that $\bar z$ satisfies \eqref{e:z_l}, in $(\bar\alpha,1)$, with $\bar D>0$ and $\tilde g>0$. Then, we can apply \cite[Proposition 3]{BCM2}, but with opposite sign speed, because of \eqref{e:z_l}: we deduce that there exists $c^*_{\alpha0}> 0$ (arguing as above) such that $\bar z(\bar\alpha)=0$, that is $z(\alpha)=0$, if and only if $c \le -c^*_{\alpha0}$.  Therefore,
\begin{equation*}
\begin{aligned}
\mathcal F(c,\alpha)&= z_c\left(\eta(\alpha)\right)-z_c(\alpha)+c\left(\eta(\alpha)-\alpha\right)
\\
&=z_c\left(\eta(\alpha)\right)+c\left(\eta(\alpha)-\alpha\right)<0 
\quad \text{for } c \le -c_{\alpha0}^*.
\end{aligned}
\end{equation*}
Then $c^*(\alpha)>-c_{\alpha0}^*$. Since the function $c^*(\varphi_\ell)$ is strictly decreasing from Theorem~\ref{t:main}, we get $c^*(\phi_\ell)\ge c^*(\alpha)$ and so
\[
c^*(\varphi_\ell)> -c_{\alpha0}^*,
\]
for every $\varphi \in \mathcal I$. In this case, we set $c_{-}:=-c_{\alpha0}^*$.
As in the case of $c^+$, we obtain 
\begin{equation}
\label{e:est_c1}
-2\sqrt{\sup_{[0, \alpha)}\frac{1}{\alpha - \phi}\int_{\phi}^{\alpha}\frac{-g(\sigma)D(\sigma)}{\alpha-\sigma}\, d\sigma}\le -c^*_{0\alpha} \le -2\sqrt{g(\alpha)D'(\alpha)}.
\end{equation}

\medskip

Finally, \eqref{e:c1c2} is a consequence of \eqref{e:est_c2} and \eqref{e:est_c1}.
\end{proof}


\begin{remark}\label{r:skkwe}
{\rm
If $P(\beta)=P(0)$ in Proposition~\ref{prop:c*_bounds}, then $\eta(0)=\beta$ and
\[
\mathcal F(c,0)=\mathcal F\left(c,\zeta(\beta)\right)=z_c(\beta)-z_c(0)+c\beta = z_c(\beta)+c\beta.
\]
Hence, $\mathcal F(c,0) >0$ if $c \ge c^*_{1\beta}>0$ (see the proof of Proposition~\ref{prop:c*_bounds})  and $\mathcal F(c,0)<0$ if $c \le 0$; we deduce $c^*(0) \in (0, c^*_{1\beta})$. This estimate complements \eqref{e:est_c2} and shows a case when $c^*>0$. 

Analogously, if $P(1)=P(\alpha)$, then $\alpha \in \mathcal I$ and $\eta(\alpha)=1$. Hence, by Lemma \ref{l:z1z0} we deduce
\[
\mathcal F(c,\alpha)=-z_c(\alpha)+c(1-\alpha).
\]
Clearly, $\mathcal F(c,\alpha)<0$ if $c\le -c^*_{\alpha0}$, and $\mathcal F(c,\alpha)>0$ if $c \ge 0$. Hence, $c^*(\alpha) \in (-c^*_{\alpha0},0)$. This estimate complements estimate \eqref{e:est_c1} and shows a case where $c^*<0$.
}
\end{remark}

\begin{remark}
{
\rm If $P(1)>P(0)$ but either $P(\beta)<P(0)$ or $P(1)<P(\alpha)$ (or both) hold, we can only conclude what follows.

If $P(\beta)<P(0)$, then $0 \in \mathcal I$ and $\eta(0)>\beta$. From Lemma \ref{l:z1z0} we deduce $z_c(0)=0$ and therefore
\[
\mathcal F(c,0)=
z_c\left(\eta(0)\right)-z_c(0)+c\, \eta(0)=z_c\left(\eta(0)\right)+c\, \eta(0) <0, 
\quad \text{for } c \le 0.
\]
Therefore $c^*(0)>0$ and the monotonicity of $c^*$ yields $c^*(\varphi_\ell) < c^*(0)$.

If $P(\alpha)>P(1)$, then $1 \in \eta(\mathcal I)$. Since $z_c(1)=0$, we have
\[
\mathcal F\left(c, \zeta(1)\right)=
-z\left(\zeta(1)\right)+c\left(1-\zeta(1)\right) >0 \quad \text{for } c \ge 0.
\]
Hence $c^*(\zeta(1))<0$ and the monotonicity of $c^*$ implies
$c^*(\varphi_\ell) \ge c^*(\zeta(1))$.
}
\end{remark}

\subsection{Concluding remarks}\label{ss:concluding-remark}
We comment on conditions \eqref{e:hypo} by showing some similarities with hyperbolic equations. 

We saw in Proposition \ref{p:twins} that for every $c\in\R$ there exist unique regular semi-wavefronts from $1$ (reaching $\beta$) and to $0$ (reaching $\alpha$). Both semi-wavefronts where used in the proof of Theorem \ref{t:main} by restricting them to some intervals $(-\infty,\xi_s)$ and $(\xi_s,\infty)$, respectively, where $\phi(\xi_s^-)=\phi_r$ and $\phi(\xi_s^+)=\phi_\ell$, for  $\phi_r$, $\phi_\ell$ and $c$ satisfying \eqref{e:hypo}. 

Proposition \ref{p:twins} was proved by showing that the functions associated to the profiles (see \eqref{eq:z}) satisfy suitable boundary-value problems in $[\beta,1]$ and $[0,\alpha]$, respectively, see \eqref{e:z_r} and \eqref{e:z_l}; we denote by $z_c$ the solution in $[0,\alpha] \cup[\beta,1]$ of both problems with the same $c$. Notice that $z_c<0$ in $(0,\alpha) \cup(\beta,1)$ and then
\begin{equation}\label{e:Dgsign}
Dg/z_c>0\  \hbox{ in }(0,\alpha)\quad \hbox{ and } \quad Dg/z_c<0\ \hbox{ in }(\beta,1). 
\end{equation}
Assume $u=u(x,t)$ is a weak solution to the balance law
\begin{equation}\label{e:hypPDE}
u_t-z_c(u)_x=g(u),
\end{equation}
and jumps from $\phi_\ell$ to $\phi_r$. If $c$ is the propagation speed of the jump, then  \eqref{e:cz} must be satisfied \cite{Bressan, Dafermos}. In the framework of hyperbolic equations, equality \eqref{e:cz} is called Rankine-Hugoniot condition, see Figure \ref{f:cz}.
Since we recover equation \eqref{eq:z} from \eqref{e:hypPDE} if $u(x,t):=\phi(x-ct)$, then
 equation \eqref{e:hypPDE} can be interpreted as the hyperbolic counterpart of the parabolic equation \eqref{e:E}.

\begin{figure}[htb]
\begin{center}
\begin{tikzpicture}[>=stealth, scale=0.65]

\draw[->] (0,0) --  (10,0) node[below]{$\phi$} coordinate (x axis);
\draw[->] (0,0) -- (0,4.5) node[right]{$z$} coordinate (y axis);

\draw[very thick] (0,0) .. controls (1,2) and (2,3.2) .. (3,3.5) node[midway,left]{\footnotesize$-z$};
\draw[very thick] (6,3.7) .. controls (7,3) and (8,2) .. (9,0) node[midway,right]{\footnotesize$-z$} node[below]{\footnotesize$1$};
\draw[dotted] (3,0) node[below]{\footnotesize$\alpha$}--(3,3.5);
\draw[dotted] (6,0) node[below]{\footnotesize$\beta$}--(6,3.7);

\draw[dotted] (1,0) node[below]{\footnotesize$\phi_\ell$}--(1,1.8);
\draw[dotted] (7,0) node[below]{\footnotesize$\phi_r$}--(7,2.8);

\draw (1,1.8)--(7,2.82) node[midway, below]{\footnotesize$c$};
\end{tikzpicture}

\end{center}
\caption{\label{f:cz}{The geometric meaning of condition \eqref{e:cz}. The speed $c$ is the slope of the line joining the points $\left(\phi_\ell,-z(\phi_\ell)\right)$ and $\left(\phi_r,-z(\phi_r)\right)$.}}
\end{figure}
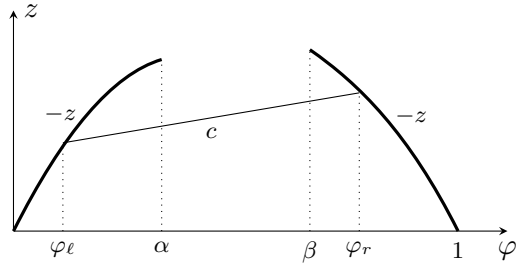

The function $z$ is defined on two non-intersecting intervals, so that the situation is similar to what happens in the modeling of phase transitions \cite{Colombo-Corli_1999}. 
Consider for instance the case of fluid-dynamics; let $v$ be the specific volume and $p$ the pressure. If the interface is stationary, then the pressure of the two boundary phases must coincide; moreover, the horizontal line in the $(v,p)$ plane must cut the plot of the pressure into two parts of equal area. This condition is similar to the equal area condition on $D$ stated above.

Here, the role of the two phases is played by the intervals $[0,\alpha]$ and $[\beta,1]$ where the diffusivity is positive. In the $(x,t)$ plane the phase boundary is the line $x=ct$ and the characteristic speeds of equation \eqref{e:hypPDE} are pointing (tangent, at most) towards the line $x=ct$. This follows by \eqref{e:Dgsign}, because by \eqref{eq:z} the characteristics speeds are $\lambda(u):=c+\frac{D(u)g(u)}{z_c(u)}$.
This means that the jump behaves like a (entropic) shock \cite{Bressan, Dafermos}.

\section*{Acknowledgments}
The authors are members of the {\em Gruppo Nazionale per l'Analisi Matematica, la Probabilit\`{a} e le loro Applicazioni} (GNAMPA) of the {\em Istituto Nazionale di Alta Matematica} (INdAM) and acknowledge financial support from this institution. D.B. was supported by MUR - M4C2 1.5 of PNRR with grant no. ECS00000036, A.C. and L.M. were supported by the PRIN 2022 project \emph{Modeling, Control and
Games through Partial Differential Equations}, funded by the European
Union - Next Generation EU, Mission 4, Component 1, CUP
D53D23005620006.

{\small
\bibliography{refe_BCM6.bib}
\bibliographystyle{abbrv2}
}

\end{document}